\documentclass[11pt, a4paper,oneside,centertags,reqno]{amsart}

\usepackage[title]{appendix}

\usepackage{etoolbox}
\robustify{\footnote}

\usepackage[multiple]{footmisc}
\usepackage{txfonts}
\usepackage{exscale}
\usepackage{amsbsy}
\usepackage[colorlinks=true,citecolor=red,pagebackref,hypertexnames=false]{hyperref}
\usepackage{lmodern}

\usepackage{endnotes}

\usepackage{latexsym}
\usepackage{amsmath}
\usepackage{amsfonts}
\usepackage{amssymb}
\usepackage{graphpap}
\usepackage{epsfig}

\usepackage{amscd}
\usepackage{amsthm}

\usepackage{enumerate}
\usepackage{eucal}

\usepackage[francais]{[babel}
\usepackage[deutsch]{[babel}
\usepackage[ansinew]{inputenc}

\usepackage{graphics}

\usepackage[backrefs,non-compressed-cites,initials,nobysame,]{amsrefs}

\usepackage{pdfsync}
\usepackage{color}

\usepackage{mathrsfs}

\usepackage{bbm}
\font\got=eufm10 at 11pt

\font\posebni=msam10

\renewcommand{\Re}[0]{{\rm Re}\,}
\renewcommand{\Im}[0]{{\rm Im}\,}

\newcommand{\2}[0]{\v{C}}

\newcommand{\ca}[0]{\mathbbm{1}}

\newcommand{\C}[0]{{\mathbb C}}

\newcommand{\f}[0]{\varphi}

\newcommand{\N}[0]{{\mathbb N}}

\newcommand{\R}[0]{\mathbb{R}}

\newcommand{\leqsim}[0]{\,\text{\posebni \char46}\,}
\newcommand{\geqsim}[0]{\,\text{\posebni \char38}\,}

\newcommand{\cA}[0]{{\mathcal A}}
\newcommand{\cB}[0]{{\mathcal B}}

\newcommand{\cE}[0]{{\mathcal E}}
\newcommand{\cF}[0]{{\mathcal F}}

\newcommand{\cI}[0]{{\mathcal I}}

\newcommand{\cM}[0]{{\mathcal M}}

\newcommand{\cP}[0]{{\mathcal P}}
\newcommand{\cQ}[0]{{\mathcal Q}}
\newcommand{\cR}[0]{{\mathcal R}}

\newcommand{\cV}[0]{{\mathcal V}}
\newcommand{\cW}[0]{{\mathcal W}}
\newcommand{\cX}[0]{{\mathcal X}}

\newcommand{\bA}{\boldsymbol{A}}

\newcommand{\oA}[0]{{\mathscr A}}

\newcommand{\oL}[0]{{\mathscr L}}

\newcommand{\oV}[0]{{\mathscr V}}

\newcommand{\sX}[0]{{\mathfrak X}}

\newcommand{\sH}[0]{{\mathfrak H}}
\newcommand{\ovR}[0]{\overline{\rm R}}

\newcommand{\gota}[0]{{\text{\got a}}}

\newcommand\bS{\mathbf{S}}

\newcommand{\mn}[2]{\{ #1 : #2 \}}
\newcommand{\Mn}[2]{\left\{ #1 : #2 \right\}}
\newcommand{\sk}[2]{\left\langle #1 , #2\right\rangle}

\renewcommand{\div}[0]{{\rm div}\,}

\newcommand{\Dom}[0]{{\rm D}}

\newcommand{\Jei}[0]{\left(\begin{array}{cc}0& -I_{\R^{d}} \\ I_{\R^{d}} & 0\end{array}\right)}

\newtheorem{theorem}{Theorem}
 
\newtheorem{lemma}[theorem]{Lemma}
\newtheorem{proposition}[theorem]{Proposition}
\newtheorem{corollary}[theorem]{Corollary}

\theoremstyle{definition}
\newtheorem{example}[theorem]{Example}

\renewcommand\leq[0]{\leqslant}
\renewcommand\geq[0]{\geqslant}
\renewcommand\epsilon[0]{\varepsilon}
\renewcommand\theta[0]{\vartheta}

\newcommand\wrt{\,\text{\rm d}}
\renewcommand\mod[1]{\left\vert{#1}\right\vert}

\newcommand\norm[2]{{\left\Vert{#1}\right\Vert_{#2}}}

\newtheorem{preremark}[theorem]{Remark}  \newenvironment{remark}%
{\begin{preremark}\rm}{\end{preremark}}

\begin{document}

\title[Bilinear embedding on domains]{Bilinear embedding for divergence-form operators with complex coefficients on irregular domains}
\author[Carbonaro]{Andrea Carbonaro}
\author[Dragi\v{c}evi\'c]{Oliver Dragi\v{c}evi\'c}

\date{July 24, 2019}

\address{Andrea Carbonaro \\Universit\`a degli Studi di Genova\\ Dipartimento di Matematica\\ Via Dodecaneso\\ 35 16146 Genova\\ Italy }
\email{carbonaro@dima.unige.it}

\address{Oliver Dragi\v{c}evi\'c \\University of Ljubljana\\ Faculty of Mathematics and Physics\\ and Institute of Mathematics, Physics and Mechanics\\ Jadranska 21, SI-1000 Ljubljana\\ Slovenia}
\email{oliver.dragicevic@fmf.uni-lj.si}

\begin{abstract}
Let $\Omega\subseteq\R^{d}$ be open and $A$ a complex uniformly strictly accretive $d\times d$ matrix-valued function on $\Omega$ with $L^{\infty}$ coefficients. Consider the divergence-form operator $\oL^{A}=-\div(A\nabla)$ with mixed boundary conditions on $\Omega$. We extend the bilinear inequality that we proved in \cite{CD-DivForm} in the special case when $\Omega=\R^{d}$. As a consequence, we obtain that the solution to the parabolic problem $u^{\prime}(t)+\oL^{A}u(t)=f(t)$, $u(0)=0$, has maximal regularity in $L^{p}(\Omega)$, for all $p>1$ such that $A$ satisfies the $p$-ellipticity condition that we introduced in \cite{CD-DivForm}.
This range of exponents is optimal for the class of operators we consider. We do not impose any conditions on $\Omega$,  in particular, we do not assume any regularity of $\partial\Omega$, nor the existence of a Sobolev embedding. The methods of \cite{CD-DivForm} do not apply directly to the present case and a new argument is needed.
\end{abstract}

\maketitle

\section{Introduction}
\label{s: Neumann introduction}

Let $\Omega\subseteq\R^{d}$ be a nonempty open set. Denote by $\cA(\Omega)$ the class of all complex uniformly strictly elliptic $d\times d$ matrix-valued functions on $\Omega$ with $L^{\infty}$ coefficients (in short, elliptic matrices). That is to say $\cA(\Omega)$ is the class of all measurable $A:\Omega\rightarrow \C^{d\times d}$ for which there exist $\lambda=\lambda(A),\Lambda=\Lambda(A)>0$ such that for almost all $x\in\Omega$ we have
\begin{eqnarray}
\label{eq: N ellipticity}
\Re\sk{A(x)\xi}{\xi}
&\hskip -19pt\geq \lambda|\xi|^2\,,
&\quad\forall\xi\in\C^{d};
\\
\label{eq: Neumann bounded}
\mod{\sk{A(x)\xi}{\sigma}}
&\hskip-6pt\leq \Lambda \mod{\xi}\mod{\sigma}\,,
&\quad\forall\xi,\sigma\in\C^{d}.
\end{eqnarray}

Suppose that $A\in\cA(\Omega)$. Fix a closed subspace $\oV$ of $W^{1,2}(\Omega)$ containing $W^{1,2}_{0}(\Omega)$. Denote by $\oL^{A}_{2}=\oL^{A,\oV}_{2}$ the unbounded operator on $L^{2}(\Omega)$ associated with the densely defined, accretive, continuous and closed sesquilinear form
$$
\gota_{A,\oV}(u,v)= \int_{\Omega}\sk{A\nabla u}{\nabla v}_{\C^{d}},\quad \Dom(\gota_{A,\oV})=\oV.
$$
Namely,
 $$
 \Dom\left(\oL^{A}_{2}\right):=\Mn{u\in \oV}
 {\exists w\in L^2(\Omega):\ 
 \gota_{A,\oV}(u,v)=\sk{w}{v}_{L^2(\Omega)}\ \forall v\in \oV}
$$
and 
\begin{equation}
\label{eq: ibp}
\int_{\Omega}\oL^{A}_{2}(u)\,\overline{v}=\gota_{A,\oV}(u,v),\quad \forall u\in\Dom(\oL^{A}_{2}),\quad \forall v\in\oV\,.
\end{equation}

Ellipticity of $A$ implies that the form $\gota_{A,\oV}$ is sectorial in the sense of Kato:
$$
\theta^{*}_{2}:=\sup\left\{\arg(\gota_{A,\oV}(u,u)): u\in \oV
\right\}<\pi/2.
$$
Therefore, the associated operator is the negative generator of a strongly continuous semigroup $(T^{A,\oV}_{t})_{t>0}$ on $L^{2}(\Omega)$ which is analytic and contractive in the cone 
$$
\bS_{\theta_{2}}=\mn{z\in\C\setminus\{0\}}{|\arg (z)|<\theta_{2}},
$$ 
where $\theta_{2}=\pi/2-\theta^{*}_{2}$. We also have $\oL^{A^{*}}_{2}=(\oL^{A})^{*}_{2}$, so $T^{A^{*}}_{t}=(T^{A}_{t})^{*}$ for all $t> 0$. For details and proofs, see \cite[Chapter VI]{Kat}, \cite{AuTc} and \cite[Chapters I and IV]{O}.
\medskip

Given a closed set $D\subseteq \partial\Omega$ we define
$$
W^{1,2}_{D}(\Omega)=\overline{\{u_{\vert_{\Omega}}: u\in C^{\infty}_{c}(\R^{d}\setminus D)\}}^{W^{1,2}(\Omega)}.
$$
\subsection{Mixed boundary conditions}\label{s: boundary}
We shall always assume that $\oV$ is one of the following closed subspaces of $W^{1,2}(\Omega)$:
\begin{enumerate}
\item $\oV=W^{1,2}(\Omega)$ corresponding to Neumann boundary conditions for $\oL^{A}$, or
\item $\oV=W^{1,2}_{D}(\Omega)$ corresponding to Dirichlet conditions in $D$ and Neumann conditions in $\partial\Omega\setminus D$ for $\oL^{A}$.
\end{enumerate}
The latter case includes 
Dirichlet boundary conditions ($D=\partial\Omega$) and  good-Neumann boundary conditions ($D=\emptyset$); see \cite[Section~4.1]{O}.
\smallskip

We notice that the very same boundary conditions have been recently considered, for example, in \cite{Egert,Egert2018,Elst2019} but under stronger assumptions on $\Omega$.
\smallskip

In the special case of pure Neumann boundary conditions $\oV=W^{1,2}(\Omega)$ we denote the semigroup generated by $-{\rm div}(A\nabla)$, $A\in\cA(\Omega)$, simply by $(T^{A}_{t})_{t>0}$.

\subsection{The $p$-ellipticity condition} 
We summarize the following notion, which we introduced in \cite{CD-DivForm}.

For every $p\in [1,+\infty]$ consider the $\R$-linear operator
$$
\cI_{p}(\xi)=\xi+(1-2/p)\overline{\xi},\quad \xi\in\C^{d}.
$$
Given $A\in\cA(\Omega)$ and $p\in [1,+\infty]$, we introduce the number
\begin{equation}
\label{eq: N Deltap}
\Delta_p(A):=
\underset{x\in\Omega}{{\rm ess}\inf}\min_{|\xi|=1}\Re\sk{A(x)\xi}{\cI_{p}\xi}_{\C^{d}}.
\end{equation}
We say that $A$ is {\it $p$-elliptic} if 
\begin{equation}
\label{eq: Delta>0}
\Delta_p(A)>0.
\end{equation}
By definition, $A$ is $p$-elliptic if and only if there exists $C=C(A,p)>0$ such that for a.e. $x\in\Omega$,
\begin{equation*}
\Re\sk{A(x)\xi}{\cI_{p}\xi}_{\C^{d}}
\geq C |\xi|^2\,,
\quad \forall\xi\in\C^{d}.
\end{equation*}
Clearly, 
$\Delta_{2}(A)>0$ is 
a reformulation of the ellipticity condition \eqref{eq: N ellipticity}. 
It follows from the definition that a bounded matrix function $A$ is real and elliptic if and only if it is $p$-elliptic for all $p>1$. For further properties of the function $p\mapsto \Delta_{p}(A)$ we refer the reader to \cite{CD-DivForm} and Lemma~\ref{l: N prop of Deltap}.

Dindo\v s and Pipher in \cite{Dindos-Pipher} and \cite{Dindos-Pipher2, Dindos-Pipher3} showed that the key condition \eqref{eq: Delta>0} also bears deep connections with the regularity theory of elliptic PDE. 
They found the sharp condition which permits proving reverse H\"older inequalities for weak solutions to ${\rm div}(A\nabla u)=0$ with complex $A$. It turns out that this condition is precisely a reformulation of $p$-ellipticity \eqref{eq: Delta>0}. 

Recently, Egert \cite{Egert2018} and, independently, ter Elst, Haller-Dintelmann, Rehberg and Tolksdorf \cite{Elst2019} used $p$-ellipticity and its properties for studying semigroup extrapolation and parabolic maximal regularity for divergence form operators with mixed boundary conditions on domains $\Omega$ that satisfy certain geometric assumptions. 

A condition similar to \eqref{eq: Delta>0}, namely $\Delta_p(A)\geq0$, was formulated in a different manner by Cialdea and Maz'ya in \cite[(2.25)]{CiaMaz}. It was a result of their study of a condition on forms known as $L^{p}$-dissipativity. 
We arrived in \cite{CD-DivForm} at the $p$-ellipticity, and thus also at $\Delta_p(A)\geq0$, from another direction (bilinear embeddings and generalized convexity of power functions) further developing and extending the methods from \cite{CD-mult} and \cite{CD-OU}; see \cite[Remark 5.9]{CD-DivForm}.

\subsection{Semigroup estimates and bilinear embedding on $\R^{d}$}\label{s: sem emb}
In \cite[Theorem~1.3]{CD-DivForm} we used a theorem of Nittka \cite[Theorem 4.1]{Nittka} and showed that the condition $\Delta_{p}(A)\geq 0$ implies contractivity in $L^{p}(\Omega)$ of the semigroup generated by $-\div(A\nabla)$ with {\it Dirichlet} boundary conditions in $\Omega$. This improved an earlier result of Cialdea and Maz'ya \cite{CiaMaz}. A straightforward modification of the proof shows that the implication $(a)\Rightarrow (b)$ in \cite[Theorem~1.3]{CD-DivForm} still holds true if we replace Dirichlet boundary conditions with Neumann boundary conditions. One can also consider mixed boundary conditions, see \cite{Egert2018}. So we have the following result.

\begin{proposition}\label{p: N Nittka} 
Suppose that $\Omega\subseteq\R^{d}$ is open. Let $\oV$ be one of the subspaces of Section~\ref{s: boundary}. Let $A\in\cA(\Omega)$ and $p>1$ be such that $\Delta_{p}(A)\geq0$. Then $(T^{A,\oV}_{t})_{t>0}$ extends to a strongly continuous semigroup of contractions on $L^{p}(\Omega)$.
\end{proposition}
One of the main points of \cite{CD-DivForm} was the connection between $p$-ellipticity and bilinear embeddings associated with divergence-form operators with complex coefficients.
More specifically, given $A,B\in\cA(\Omega)$ define
\begin{equation}
\label{eq: DplAB}
\aligned
\Delta_{p}(A,B)&=\min\{\Delta_{p}(A),\Delta_{p}(B)\},\\
\lambda(A,B)&=\min\{\lambda(A),\lambda(B)\},\\
\Lambda(A,B)&=\max\{\Lambda(A),\Lambda(B)\}.
\endaligned
\end{equation}
In the special case $\Omega=\R^{d}$ we proved in \cite[Theorem 1.1]{CD-DivForm} that there exists $C>0$, depending only on $p$, $\Delta_{p}(A,B)$, $\lambda(A,B)$ and $\Lambda(A,B)$, such that
\begin{equation}\label{eq: N bil RN}
\Delta_{p}(A,B)>0
\hskip 5pt
\Longrightarrow 
\hskip 5pt
\int^{\infty}_{0}\!\int_{\R^{d}}\mod{\nabla T^{A}_{t}f(x)}\mod{\nabla T^{B}_{t}g(x)}\wrt x\wrt t\leq C\norm{f}{p}\norm{g}{q},
\end{equation}
for all $f,g\in (L^{p}\cap L^{q})(\R^{d})$.
\medskip

The $p$-ellipticity condition appeared while we were studying in \cite{CD-DivForm} the validity of the right-hand side of \eqref{eq: N bil RN} and is related to the notion of generalised convexity, or convexity with respect to matrices, that we previously studied in some specific cases in \cite{CD-mult} and \cite{CD-OU} and that we shall discuss in Section~\ref{s: Neumann gen conv}.

\subsection{Bilinear embedding on domains}\label{s: N bill dom} 

While Proposition~\ref{p: N Nittka} and \cite[Theorem~1.3]{CD-DivForm} hold true for all open $\Omega\subseteq\R^{d}$, in \cite{CD-DivForm} we were able to prove the bilinear estimate \eqref{eq: N bil RN} only in the special case $\Omega=\R^{d}$. One of the targets of the present paper is to extend \eqref{eq: N bil RN} to every open set $\Omega\subset\R^{d}$. In Section~\ref{s: N proof bil} we shall prove the following result.
\begin{theorem}\label{t: N bil}
Suppose that $\Omega\subset\R^{d}$ is an open set. Let $\oV$ and $\oV^{\prime}$ be two closed subspaces of $W^{1,2}(\Omega)$ of the type described in Section~\ref{s: boundary}. Let $p>1$, $q=p/(p-1)$ and $A,B\in\cA(\Omega)$. Then there exists $C>0$, depending only on $p$, $\Delta_{p}(A,B)$, $\lambda(A,B)$ and $\Lambda(A,B)$, such that
\begin{equation}\label{eq: N bil}
\Delta_{p}(A,B)>0
\hskip 5pt
\Longrightarrow 
\hskip 5pt
\int^{\infty}_{0}\!\int_{\Omega}\mod{\nabla T^{A,\oV}_{t}f(x)}\mod{\nabla T^{B,\oV^{\prime}}_{t}g(x)}\wrt x\wrt t\leq C \norm{f}{p}\norm{g}{q},
\end{equation}
for all $f,g\in (L^{p}\cap L^{q})(\Omega)$.
\end{theorem}
The method we used in \cite{CD-DivForm} for proving \eqref{eq: N bil RN} does not apply to arbitrary open $\Omega\subset\R^{d}$, even for pure Neumann boundary conditions $\cV=\cV^{\prime}=W^{1,2}(\Omega)$. Indeed, we proved \eqref{eq: N bil RN} by means of a regularisation argument \cite[Section~6]{CD-DivForm} which reduces the proof to the case of smooth $A, B\in\cA(\R^{d})$ with bounded derivatives. The reduction procedure was used for justifying the integration by parts behind the formula \cite[(3.3)]{CD-DivForm}; see \cite[Section~4.2]{CD-DivForm}. In $\R^{d}$ the advantage of working with smooth coefficients with bounded derivatives is that by results of Auscher, McIntosh and Tchamitchian \cite[Theorem~4.15]{AMT} and Auscher \cite[Theorem~4.8]{AuscherNuc}, in this case, $T^{A}_{t}$ and $T^{B}_{t}$ are bounded in $L^{\infty}(\R^{d})$, for all $t>0$; this was the property we used for proving \cite[Theorem~1.1]{CD-DivForm}.

In Section~\ref{s: N bil RN} we shall simplify the proof of \cite[Theorem~1.1]{CD-DivForm} by means of a new argument based on the aforementioned regularisation trick and elliptic regularity \cite{AgDoNi} for smooth coefficient operators. The key fact here is that if $A$ is sufficiently regular then the domain of $\oL^{A}$ in $L^{r}(\R^{d})$ coincides with $W^{2,r}(\R^{d})$, $1<r<+\infty$. This makes it possible to work with the operator core $C^{\infty}_{c}(\R^{d})$, so that all the integrations by parts can be easily justified.

For divergence-form operators $\oL^{A}$ on $\Omega\subset\R^{d}$ with, say, Neumann boundary conditions $\oV=W^{1,2}(\Omega)$ the situation is different. On one hand the domain of the Neumann Laplacian $\oL^{I}$ in $L^{p}(\Omega)$ is unknown and, in general, it is not included in $W^{2,p}(\Omega)$ \cite{Gris,Costabel} or even in $W^{1,p}(\Omega)$ \cite{JeKe}. One the other hand, extrapolation of $T^{A}_{t}$ on $L^{\infty}(\Omega)$ is not expected, even for complex constant $A$ (see Section~\ref{s: extrapolation ex}) and it is not clear if there exists an operator core of bounded functions for $\oL^{A}_{p}$, $1<p<\infty$.
This makes the regularisation procedure used in \cite{CD-DivForm} and Section~\ref{s: N bil RN} useless for the proof of \eqref{eq: N bil}, and forces us to modify the Bellman-function-heat-flow method we used in \cite{CD-DivForm}; see Section~\ref{s: Thm bil discussion}. This is the main technical novelty of the present paper.

\subsection{Maximal regularity and functional calculus on domains} By means of the elementary properties of the function $(p,A)\mapsto \Delta_{p}(A)$ (see Lemma~\ref{l: N prop of Deltap}), a general result of Cowling, Doust, McIntosh, and Yagi \cite[Theorem~4.6 and Example~4.8]{CDMY} and the Dore-Venni theorem \cite{DoreVenni,PrussSohr} we deduce from Theorem~\ref{t: N bil} applied with $B=A^{*}$ and $\oV^{\prime}=\oV$ the following result; see Section~\ref{s: max funct}.

\begin{theorem}\label{t: N principal}
Suppose that $\Omega\subseteq\R^{d}$ is an open set. Let $\oV$ be one of the subspaces of Section~\ref{s: boundary}. Let $A\in \cA(\Omega)$ and $p>1$. Suppose that $A$ is $p$-elliptic (that is
$\Delta_{p}(A)>0$). Then the negative generator $\oL^{A}_{p}$ of $(T^{A,\oV}_{t})_{t>0}$ on $L^{p}(\Omega)$ admits a bounded holomorphic functional calculus of angle $<\pi/2$. As a consequence, $\oL^{A}_{p}$ has parabolic maximal regularity.
\end{theorem}
The novelty of Theorem~\ref{t: N principal} lies in the fact that we are able to prove parabolic maximal regularity for some $p\neq2$ {\sl without assuming} any regularity of the boundary $\partial\Omega$, {\sl nor} the existence of a Sobolev embedding
\begin{equation}\label{eq: N S embedding}
\tag{${\rm SE}_{q}$}
\oV\hookrightarrow L^{q}(\Omega),
\end{equation}
for some $q>2$. Hence our results complement those of \cite{Egert2018, Elst2019}; see Section~\ref{s: comparison}.
\smallskip

Two results of Kunstmann \cite[Example~2.4 and Remark~3.2]{Kunstmann3} and \cite{kunstmann2} show that the range of $p$'s in Theorem~\ref{t: N principal} is optimal even for the class $\{\oL^{A}: A\in\cA(\Omega)\cap C^{\infty}(\Omega), 
\ \oV=W^{1,2}(\Omega), |\Omega|<+\infty\}$. This means that, given $d\geq 2$, there exist $\Omega_{0}\subset\R^{d}$ of finite measure and $A_{0}\in\cA(\Omega)$ with smooth coefficients such that $\oL^{A_{0}}$ subject to Neumann boundary conditions in $\Omega_{0}$ has parabolic maximal regularity in $L^{p}(\Omega_{0})$ if and only if $\Delta_{p}(A_{0})>0$; see Section~\ref{s: sharp}.

\section{Heat-flow monotonicity and generalised convexity}\label{s: Neumann gen conv} 
For proving the bilinear inequality \eqref{eq: N bil} we use a variant of the Bellman-function-heat-flow method originally introduced by Petermichl and Volberg in \cite{PV} and Nazarov and Volberg in \cite{NV}, and extended in \cite{DV-Sch, Dv-kato,CD-Riesz}. Here we further refine the ``complex-time'' version of this method that we developed in \cite{CD-mult, CD-OU, CD-DivForm}. This new refinement addresses a major technical issue (see Section~\ref{s: Thm bil discussion}).

\subsection{The Bellman function of Nazarov and Treil}
In the context of the present paper this method consists of studying the monotonicity of the flow\footnote{In the definition of $\cE$ we implicitly identify $T^{A,\oV}_{t}f$ and $T^{B,\oV^{\prime}}_{t}g$ with their real counterparts; see Section~\ref{s: N heat-flow} for a more accurate statement.} 
\begin{equation}\label{eq: N flow}
\cE(t)=\int_{\Omega}\cQ(T^{A,\oV}_{t}f,T^{B,\oV^{\prime}}_{t}g)
\end{equation}
associated with a particular explicit {\it Bellman function} $\cQ$ invented by Nazarov and Treil \cite{NT} in 1995. Here we use a simplified variant introduced in \cite{Dv-kato} which comprises only two variables:
\begin{equation}\label{eq: N Bellman}
\cQ(\zeta,\eta)=
|\zeta|^p+|\eta|^{q}+\delta
\begin{cases}
 |\zeta|^2|\eta|^{2-q},& |\zeta|^p\leqslant |\eta|^q;\\
 (2/p)\,|\zeta|^{p}+\left(
 2/q-1\right)|\eta|^{q},&|\zeta|^p\geqslant |\eta|^q\,,
\end{cases}
\end{equation}
where $p>2$, $q=p/(p-1)$, $\zeta,\eta\in\R^{2}$ and $\delta>0$ is a positive parameter that will be fixed later. Recall from \cite{CD-DivForm} that $\cQ\in C^{1}(\R^{4})\cap C^{2}(\R^{4}\setminus\Upsilon)$, where
$$
\Upsilon=\{\eta=0\}\cup\{|\zeta|^p=|\eta|^q\}\,.
$$
For $(\zeta,\eta)\in\R^2\times\R^2$ we have
\begin{equation}
\label{eq: N 5}
\aligned
0\leqslant \cQ(\zeta,\eta) & \leqsim_{p,\delta}\,\left(|\zeta|^p+|\eta|^q\right),\\
|(\partial_{\zeta}\cQ)(\zeta,\eta)| & \leqsim_{p,\delta}\, \max\{|\zeta|^{p-1},|\eta|\},\\
|(\partial_{\eta}\cQ)(\zeta,\eta)| & \leqsim_{p,\delta}\, |\eta|^{q-1}\,,
\endaligned
\end{equation}
where $\partial_\zeta=\left(\partial_{\zeta_1}-i\partial_{\zeta_2}\right)/2$ and $\partial_\eta=\left(\partial_{\eta_1}-i\partial_{\eta_2}\right)/2$.
\medskip

The construction of the original Nazarov--Treil function was one of the earliest examples of the so-called Bellman function technique, which was introduced in harmonic analysis shortly beforehand by Nazarov, Treil and Volberg \cite{NTV}.  
The name ``Bellman function'' stems from the stochastic optimal control, see \cite{NTV1} for details. The same paper \cite{NTV1} explains the connection between the Nazarov--Treil--Volberg approach and the earlier work of Burkholder on martingale inequalities; see \cite{Bu1} and \cite{Bu2,Bu4}. For an in-depth treatise on recent advances in martingale inequalities the reader is referred to \cite{Os}.
If interested in the genesis of Bellman functions and the overview of the method, the reader is also referred to \cite{V,NT,W}. Recent applications of Bellman-heat-flow methods include \cite{DomPe,PSW,CD-Riesz, CD-mult, CD-OU, CD-DivForm, MaSp, Dah, Wrobel1, BDFS,DV, DV-Sch,DraVol}.
\medskip

A formal passage of the time derivative under the integral sign in \eqref{eq: N flow} and a more delicate formal integration by parts (see the discussion in Section~\ref{s: N bill dom}) suggest that the monotonicity of $\cE$ is related to the convexity properties of $\cQ$; see Section~\ref{s: N heat-flow}. Indeed, it naturally leads to a new notion of convexity called {\it generalised convexity with respect to the matrices $(A,B)$}; in short {\it $(A,B)$-convexity} \cite{CD-mult,CD-OU,CD-DivForm}.

Owing to the tensor structure of $\cQ$, the generalised convexity of $\cQ$ is related to that of its elementary building blocks (see \cite{CD-DivForm} and Section~\ref{s: N powers}): the power functions
$$
F_{r}(\gamma)=|\gamma|^{r},\quad \gamma\in\R^{2},\quad r>0.
$$
It turns out  that $F_{p}$ is $A$-convex if and only if $\Delta_{p}(A)\geq 0$, and $\cQ$ is strictly $(A,B)$-convex provided that $\Delta_{p}(A,B)>0$; see \cite{CD-DivForm} and Theorem~\ref{t: N B gen conv}.
\medskip

We now formalise the notion of generalised convexity.

\subsection{Real form of complex operators}We explicitly identify $\C^{d}$ with $\R^{2d}$ as follows. 
For each $d\in\N_{+}$ consider the operator
$\cV_{d}:\C^{d}\rightarrow\R^{d}\times\R^{d}$, defined by 
$$
\cV_{d}(\xi_{1}+i\xi_{2})=
(\xi_{1},\xi_{2}),\quad \xi_{1},\xi_{2}\in\R^{d}.
$$
Let $k,d\in\N_{+}$. We define another identification operator
$$
\cW_{k,d}:\underbrace{\C^{d}\times\cdots\times\C^{d}}_{k-{\rm times}}\longrightarrow \underbrace{\R^{2d}\times\cdots\times\R^{2d}}_{k-{\rm times}},
$$
 by the rule
$$
\cW_{k,d}(\xi^{1},\dots,\xi^{k})
=\left(\cV_{d}(\xi^{1}),\dots,\cV_{d}(\xi^{k})\right),\quad \xi^{j}\in\C^{d},\ j=1,\dots,k.
$$
Denote by $J$ the standard symplectic operator on $\R^{2d}$ given by
$$
J=\Jei.
$$
The operator $J$ is associated to the standard complex structure on $\R^{d}\times\R^{d}$. Namely $J$ is the real form of the multiplication by $i$: $\cV_{d}(i\xi)=J\cV_{d}(\xi)$ and

\begin{equation*}
\sk{\xi}{\xi^{\prime}}_{\C^{d}}=\sk{\cV_{d}(\xi)}{\cV_{d}(\xi^{\prime})}_{\R^{2d}}+i\,\sk{\cV_{d}(\xi)}{J\cV_{d}(\xi^{\prime})}_{\R^{2d}},\quad \xi,\xi^{\prime}\in\C^{d}.
\end{equation*}
If $A\in\C^{d\times d}$ we shall frequently use its real form:
$$
\cM(A)=\cV_{d}A\cV_{d}^{-1}=\left(
\begin{array}{rr}
\Re A  & -\Im A\\
\Im A  & \Re A
\end{array}
\right)\,.
$$
Observe that $\cM(A^*)=\cM(A)^T$ and $\cM(AB)=\cM(A)\cM(B)$. For $\xi,\sigma\in\C^{d}$ we have 
\begin{equation*}
\sk{A\xi}{\sigma}_{\C^{d}}=\sk{\cM(A)\cV_{d}(\xi)}{\cV_{d}(\sigma)}_{\R^{2d}}+i\sk{J^{t}\cM(A)\cV_{d}(\xi)}{\cV_{d}(\sigma)}_{\R^{2d}},
\end{equation*}
In particular,
\begin{equation}\label{eq: real form}
\aligned
\Re\sk{A\xi}{\sigma}_{\C^{d}}=
\sk{\cM(A)\cV_{d}(\xi)}{\cV_{d}(\sigma)}_{\R^{2d}}.
\endaligned
\end{equation}
\subsection{Convexity with respect to complex matrices}
Let $k\in\N_{+}$ and $V\subseteq \R^{2k}$ a open subset. Suppose that $\Phi:V\rightarrow \R$ is of class $C^{2}$. Denote by $(D^{2}\Phi)(\omega)$ the Hessian matrix of $\Phi$ at the point $\omega\in V$. Let $d\in\N_{+}$ and $A_{1},\dots,A_{k}\in \C^{d\times d}$. Denote by $\cM(A^{*}_{1})\oplus\cdots\oplus\cM(A^{*}_{k})$ the $2kd\times 2kd$ block diagonal real matrix with the $2d\times 2d$ blocks $\cM(A^{*}_{1}),\dots,\cM(A^{*}_{k})$ along the main diagonal. For each $\omega\in V$, we define the new matrix
$$
H^{(A_{1},\dots,A_{k})}_{\Phi}(\omega)=\left(\cM(A^{*}_{1})\oplus\cdots\oplus\cM(A^{*}_{k})\right)\cdot\left((D^{2}\Phi)(\omega)\otimes I_{\R^{d}}\right).
$$
We call $H^{(A_{1},\dots,A_{k})}_{\Phi}(\omega)$ the {\it generalised Hessian} of $\Phi$ at the point $\omega$ with respect to the complex matrices $A_{1},\dots,A_{k}$. We say that  $\Phi$ is $(A_{1},\dots,A_{k})$-convex in $V$ if the quadratic form associated with $H^{(A_{1},\dots,A_{k})}_{\Phi}(\omega)$ is nonnegative at every $\omega\in V$. We shall often say that $\Phi$ is $(A_{1},\dots,A_{k})$-convex in a single point $\omega\in V$, if the condition above holds for that particular $\omega$. The same for $(A_{1},\dots, A_{k})$-convexity in a subset of $V$.

In accordance with \cite{CD-DivForm}, we introduce a special notation for denoting the quadratic form associated with the generalised Hessians. Given $A_{1},\dots,A_{k}\in\C^{d\times d}$ and $\Phi:V\subseteq\R^{2k}\rightarrow \R$, we define
$$
H^{(A_{1},\dots,A_{k})}_{\Phi}[\omega;\Xi]=
 \sk{H^{(A_{1},\dots,A_{k})}_{\Phi}(\omega)\Xi}{\Xi}_{\R^{2kd}},\quad \omega\in\R^{2k},\quad \Xi\in\R^{2kd}.
$$

We maintain the same notation when instead of matrices we consider matrix-valued functions $A_{1},\dots,A_{k}\in L^{\infty}(\Omega;\C^{d\times d})$; in this case however we require that all the conditions are satisfied for a.e. $x\in\Omega$.

\subsection{Heat-flow monotonicity}\label{s: N heat-flow} The main reason for introducing the notion of convexity with respect to complex matrices (generalised convexity) is its link with the monotonicity of certain functionals associated with semigroups \cite{CD-mult,CD-OU,CD-DivForm}. In what follows we explain this link at a formal level. In the applications, the justification of the formal passages is part of the problem (see the discussion in Section~\ref{s: Thm bil discussion}), and it will not be addressed here. 

Let $\Omega\subseteq\R^{d}$, $A_{1},\dots,A_{k}\in\cA(\Omega)$ and $\oV_{1},\dots,\oV_{k}$ of the type described in Section~\ref{s: boundary}. Let $\Phi: \R^{2k}\rightarrow \R_{+}$ be of class, say, $C^{2}$. Define $\Phi_{\cW}=\Phi\circ \cW_{k,1}$. Given $f_{1},\dots,f_{k}\in L^{2}(\Omega)$, define the function
$$
\cE(t)=\int_{\Omega}\Phi_{\cW}\left(T^{A_{1},\oV_{1}}_{t}f_{1},\dots,T^{A_{k},\oV_{k}}_{t}f_{k}\right),\quad t>0. 
$$
\begin{itemize}
\item[a)] Suppose that we can differentiate and interchange derivative and integral. Then a calculation (see \cite{CD-DivForm}) shows that
\begin{equation}\label{eq: N flow derivative}
-\cE^{\prime}(t)=\int_{\Omega}2\,\Re\left[\sum^{k}_{j=1}(\partial_{\zeta^{j}}\Phi)_{\cW}\left(T^{A_{1},\oV_{1}}_{t}f_{1},\dots,T^{A_{k},\oV_{k}}_{t}f_{k}\right)\oL^{A_{j}}T^{A_{j},\oV_{j}}_{t}f_{j}\right],
\end{equation}
where $\zeta^{j}=(\zeta^{j}_{1},\zeta^{j}_{2})\in\R^{k}\times\R^{k}$ and $\partial_{\zeta^{j}}=\left(\partial_{\zeta^{j}_1}-i\partial_{\zeta^{j}_2}\right)/2$.
\item[b)] Suppose that each $(\partial_{\zeta^{j}}\Phi)_{\cW}(T^{A_{1}}_{t}f_{1},\dots,T^{A_{k}}_{t}f_{k})$ belongs to the form domain $\oV_{j}$. Then we can integrate by parts in the sense of \eqref{eq: ibp} on the right-hand side of \eqref{eq: N flow derivative} and, by means of another calculation (see \cite{CD-DivForm}), we get
$$
-\cE^{\prime}(t)=\int_{\Omega}H^{(A_{1},\dots,A_{k})}_{\Phi}\left[\cW_{k,1}\left(T^{A_{1},\oV_{1}}_{t}f_{1},\dots,T^{A_{k},\oV_{k}}_{t}f_{k}\right);\cW_{k,d}\left(\nabla T^{A_{1},\oV_{1}}_{t}f_{1},\dots,\nabla T^{A_{k},\oV_{k}}_{t}f_{k}\right)\right].
$$
\end{itemize}
It follows that if $\Phi$ is $(A_{1},\dots,A_{k})$-convex on $\R^{2k}$ then the function $\cE$ is nonincreasing on $(0,+\infty)$.

When $\Phi$ is {\it strictly} $(A_{1},\dots,A_{k})$-convex and satisfies a suitable size estimate, this formal method can be used for proving bilinear inequalities in the spirit of \cite{CD-DivForm,CD-mult,CD-OU}.

\subsection{Generalised convexity of power functions and the Bellman function of Nazarov and Treil}
\label{s: N powers} 
Let $\Omega\subseteq\R^{d}$ an open subset, $A\in\cA(\Omega)$ and $p>1$. Recall the definition of the number $\Delta_{p}(A)$ in \eqref{eq: N Deltap}. 

The following facts will be used in this paper. They were proven in \cite{CD-DivForm}. 
\begin{lemma}{\cite{CD-DivForm}}
\label{l: N prop of Deltap}
Let $A\in\cA(\Omega)$, $p>1$ and $q=p/(p-1)$. Then
\begin{enumerate}[{\rm (i)}]
\item 
\label{L4 (i)}
$\Delta_{p}(A)=\Delta_{q}(A)$
\item
\label{L4 (ii)}
$\Delta_{p}(A)\geq 0$ if and only if $\Delta_{p}(A^{*})\geq 0$.
The same holds for strict inequalities.
\item 
\label{L4 (iii)}
The function $p\mapsto\Delta_p(A)$ is Lipschitz continuous and nonincreasing on $[2,\infty)$.
\item
\label{L4 (iv)}
 The function $\varphi\mapsto\Delta_{p}(e^{i\varphi}A)$ is Lipschitz continuous in the interval $(-\pi/2,\pi/2)$.
\item 
\label{L4 (v)}
For a.e. $x\in\Omega$ and every $\zeta\in\R^{2}\setminus\{0\}$ and $\xi\in\C^{d}$ we have
\begin{equation*}
H^{A(x)}_{F_{p}}[\zeta;\cV_{d}(\xi)]=\frac{p^{2}}{2}|\zeta|^{p-2}\Re\sk{A(x)e^{-i\arg(\zeta)}\xi}{\cI_{p}e^{-i\arg(\zeta)}\xi}\,.
\end{equation*}
In particular,
$$
\underset{x\in\Omega}{{\rm ess}\inf}\min_{|X|=1}H^{A(x)}_{F_{p}}[\zeta;X]= \frac{p^{2}}{2}|\zeta|^{p-2}\Delta_{p}(A),
$$
for all $\zeta\in\R^{2}\setminus\{0\}$.
\item 
\label{L4 (vi)}
$\Im A(x)=0$ for a.e. $x\in\Omega$  if and only if $\Delta_{p}(A)\geq 0$ for all $p>1$.
\end{enumerate}
\end{lemma}
\begin{remark}
In principle, Proposition~\ref{p: N Nittka} can be deduced by combining Lemma~\ref{l: N prop of Deltap}~(\ref{L4 (v)}) with the heat-flow method of Section~\ref{s: N heat-flow} applied with $\Phi=F_{p}$.  However this is just a formal argument and can not be directly used. This problem can be fixed by using a result of Nittka \cite[Theorem~4.1]{Nittka}; see \cite[Theorem~1.3]{CD-DivForm} and the beginning of Section~\ref{s: sem emb}.
\end{remark}
Recall the notation \eqref{eq: DplAB}. 
The next theorem establishes a link between 
generalized convexity of power functions $|\zeta|^p$ and $|\eta|^q$ and {\it strict} generalised convexity of the Bellman function $\cQ=\cQ_{p,\delta}$ defined in \eqref{eq: N Bellman}. 

\begin{theorem}[{\cite[Theorem 5.2]{CD-DivForm}}]
\label{t: N B gen conv}
Suppose that $p\geq2$ and $A,B\in \cA(\Omega)$ satisfy $\Delta_p(A,B)>0$. 
Then there exists $\delta=\delta
(\Delta_p(A,B),\lambda(A,B),\Lambda(A,B))
\in(0,1)$ such that the function $\cQ$ is strictly $(A,B)$-convex in $\R^{4}\setminus\Upsilon$. More specifically, for almost every $x\in\Omega$ we have
$$
H_{\cQ}^{(A(x),B(x))}[\omega;(X,Y)]
\geq \frac{\Delta_p(A,B)}{5}\cdot\frac{\lambda(A,B)}{\Lambda(A,B)}|X||Y|\,,
$$
for any $\omega\in\R^{4}\setminus\Upsilon$ and $X,Y\in\R^{2d}$.
\end{theorem}

\section{Strategy for proving the bilinear embedding (Theorem~\ref{t: N bil})}
\label{s: Thm bil discussion} 
Let $p>1$ and $A,B\in\cA(\Omega)$ be such that $\Delta_{p}(A,B)>0$. To simplify the exposition, in this section we only consider pure Neumann boundary conditions $\oV=\oV^{\prime}=W^{1,2}(\Omega)$ for $\oL^{A}$ and $\oL^{B}$. Ignore for one moment that the Bellman function $\cQ$ defined in \eqref{eq: N Bellman} is not globally $C^{2}$ (this can be easily fixed by means of convolution with smooth approximation of identity; see Section~\ref{s: conv}). 

We would like to use the heat-flow method of Section~\ref{s: N heat-flow}
applied with $k=2$ and $\Phi=\cQ$ to deduce Theorem~\ref{t: N bil} from the strict $(A,B)$-convexity of $\cQ$ (see Theorem~\ref{t: N B gen conv}), Proposition~\ref{p: N Nittka} and the first size estimate in \eqref{eq: N 5}. This was our approach in \cite{CD-DivForm}. The major difficulty here is that it is not clear whether $(\partial_{\zeta}Q)(T^{A}_{t}f,T^{B}_{t}g)$ and $(\partial_{\eta}Q)(T^{A}_{t}f,T^{B}_{t}g)$ belong to the form domain $W^{1,2}(\Omega)$ whenever $f,g\in (L^{p}\cap L^{q})(\Omega)$, so the hypothesis of Section~\ref{s: N heat-flow} b) may not be satisfied and we cannot justify the integration by parts \eqref{eq: ibp} on the right-hand side of \eqref{eq: N flow derivative}. As we remarked in the introduction, in the special case when $\Omega=\R^{d}$ we overcame this difficulty by using a regularisation argument that we learnt from \cite[Section~1.2]{AuTc}; see \cite[Theorem~1.1, Section~4.2 and Appendix]{CD-DivForm}. Since we do not see how to modify the regularisation trick in the case of general open subsets $\Omega\subset\R^{d}$, for proving Theorem~\ref{t: N bil} we instead modify the Bellman-heat-flow argument used in \cite{CD-DivForm}; see Section~\ref{s: N proof bil}.
\smallskip

Our idea is to approximate the Nazarov-Treil Bellman function $\cQ$ with a sequence $(\cR_{n,\nu})_{n\in\N_{+}}$, $\nu>0$, of smooth $(A,B)$-convex functions having first order partial derivatives of linear growth and bounded second order derivatives (see Theorem~\ref{t: N approx}), in a such a way that, for $\Phi=\cR_{n,\nu}$, the integration by parts in \eqref{eq: N flow derivative} is justified. Then we deduce Theorem~\ref{t: N bil} by a limiting argument. The construction of the sequence $(\cR_{n,\nu})_{n\in\N_{+}}$, although based on elementary arguments, requires some effort, because of the rigidity of the $(A,B)$-convexity. 
\medskip

We now shortly describe the main steps in the construction of the sequence $(\cR_{n,\nu})_{n\in\N_{+}}$. The technical details are postponed to Section~\ref{s: the sequence}. Denote by $\star$ the convolution in $\R^4$. 
\begin{enumerate}[(a)]
\item Since we need to use the chain-rule for the composition with vector-valued Sobolev functions \cite{Leonichain, AmbMaschain}, it is convenient to replace $\cQ$ with its regularised version $\cQ\star\varphi_{\nu}$, where $\{\varphi_{\nu}\}_{\nu\in (0,1)}$ is a smooth compactly supported approximation of the identity in $\R^{4}$; see Section~\ref{s: conv}.
\item The size of the first- and second-order derivatives of $\cQ\star\varphi_{\nu}$ (see the estimates of Lemma~\ref{l: N second order}) does not justify the integration by parts as in Section~\ref{s: N heat-flow} with $\Phi=\cQ\star\varphi_{\nu}$. So we cut $\cQ\star\varphi_{\nu}$ by means of a sequence of smooth mollifiers $\{\psi_{n}\}_{n\in\N_{+}}$ such that $\psi_{n}\geq 0$ is supported, say, in $B_{\R^{4}}(0,4n)$ and $\psi_{n}=1$ in $B_{\R^{4}}(0,3n)$.
\item The function $\psi_{n}\cdot (\cQ\star\varphi_{\nu})$ is not $(A,B)$-convex in the region $\{\omega\in\R^{4}: 3n\leq |\omega|\leq 4n\}$. To fix this problem, we add another regular function $\cP_{n,\nu}$ which is globally $(A,B)$-convex and strictly $(A,B)$-convex in the annulus, in order that 
$$
\cR_{n,\nu}=\psi_{n}\cdot(\cQ\star\varphi_{\nu})+\cP_{n,\nu}
$$ 
becomes $(A,B)$-convex in all $\R^{4}$; see Section~\ref{s: conv}.
\end{enumerate}
In the construction of $\cP_{n,\nu}$ we need to bear in mind that: 
\begin{itemize}
\item For each $n$ the function $\cR_{n,\nu}$ must have  partial derivatives with linear growth and bounded second order derivatives (needed for the integration by parts as in Section~\ref{s: N heat-flow} b) with $\Phi=\cR_{n,\nu}$);
\item $D^{2}\cR_{n,\nu}$ must converge pointwise to $D^{2}(\cQ\star\varphi_{\nu})$ (needed for applying the strict $(A,B)$-convexity of $\cQ\star\varphi_{\nu}$; see Proposition~\ref{p: N gen conv reg}); 
\item $D\cR_{n,\nu}$ must converge pointwise to $D(\cQ\star\varphi_{\nu})$ and satisfy $$\mod{(D\cR_{n,\nu})(\omega)}\leqsim_{\nu} \left(|\omega|^{p-1}+|\omega|^{q-1}
\right),\quad \forall \omega\in\R^{4}$$
uniformly in $n\in\N_{+}$ (needed for applying Lebesgue convergence theorem and passing to the limit as $n$ goes to $\infty$ in the right-hand side of \eqref{eq: N flow derivative} with $\Phi=\cR_{n,\nu}$).
\end{itemize}
\medskip

We shall define 
\begin{equation}\label{eq: explPn}
\cP_{n,\nu}:=C(\nu)\cP_{n}\star\varphi_{\nu},
\end{equation}
for a suitable sequence $\{\cP_{n}\}_{n\in\N_{+}}$ and constant $C(\nu)>0$; see Sections~\ref{s: the sequence}~and~\ref{s: conv}. 
\medskip

Since in the ball $\{\omega\in\R^{4}: |\omega|\leq 5n\}$ we have the estimate
$$
\mod{(D^{2}(\psi_{n}\cdot (\cQ\star\varphi_{\nu}))(\omega)}\leqsim_{\nu}n^{p-2};
$$
(see \eqref{eq: N annulus}), a natural choice (see, for example, \cite{Yan}) for the sequence $\{\cP_{n}\}_{n\in\N_{+}}$ would be $\cP_{n}(\omega):=P_{n}(\omega):=p_{n}(|\omega|)$, where
$$
p_{n}(t):=C(\nu)\,n^{p-2}
\begin{cases}
0,&t\in [0,n];\\
(t-n)^{2},& t>n.
\end{cases}
$$
Unfortunately, Proposition~\ref{p: limit of g-conv} applied with $\Gamma(\zeta)=P_{n}(\zeta,0)$ or $\Gamma(\zeta)=P_{n}(0,\zeta)$ for $\zeta\in\R^{2}$, shows that if either $\Im A\neq 0$ or $\Im B\neq 0$, the function $P_{n}$ is {\bf not} $(A,B)$-convex in $\R^{4}\setminus\{|\omega|=n\}$.
\medskip

By Lemma~\ref{l: N prop of Deltap}~(\ref{L4 (v)}), under the assumption $\Delta_{p}(A,B)>0$, the sum of the $2$-variable power functions
$$
G_{p}(\zeta,\eta):=F_{p}(\zeta)+F_{p}(\eta),\quad \zeta,\eta\in\R^{2}
$$
is $(A,B)$-convex in $\R^{4}\setminus\{0\}$ and by Lemma~\ref{l: N prop of Deltap}~(\ref{L4 (iii)}) the range of $p$-ellipticity is open. So in \eqref{eq: explPn} one can try to take $\cP_{n}$ of the form
$$
\cP_{n}(\zeta,\eta):=f_{n}(|\zeta|)+f_{n}(|\eta|),\quad \zeta,\eta\in\R^{2},
$$
where
\begin{equation}\label{eq: deffn}
f_{n}(t):=\begin{cases}
n^{-\epsilon}t^{p+\epsilon},& 0\leq t\leq n,\\
\frac{p+\epsilon}{2}n^{p-2}t^{2}+(1-\frac{p+\epsilon}{2})n^{p},&t\geq n
\end{cases}
\end{equation}
and $\epsilon>0$ is such that $\Delta_{p+\epsilon}(A,B)>0$.
However, this is {\bf not} enough to compensate for the lack of $(A,B)$-convexity of $\psi_{n}\cdot (\cQ\star\varphi_{\nu})$ (see point (c)) in regions of the form $\{\omega\in\R^{4}:3n\leq|\omega|\leq 4n\}\cap S_{\kappa}$, $\kappa>1$, where
$$
S_{\kappa}:=\left\{ |\zeta|\leq\kappa^{-1}|\eta|\right\}\cup\left\{ |\eta|\leq\kappa^{-1}|\zeta|\right\}.
$$ 
In light of the previous considerations, one can try to define $\cP_{n}$ by means of suitable truncated $4$-variable power functions of the form
$$
\cP_{n}(\omega)=f_{n}(|\omega|),\quad \omega\in\R^{4},
$$
where $f_{n}$ is given by \eqref{eq: deffn}. It turns out that even this is {\bf not} the right sequence, since, in general, the condition $\Delta_{p}(A,B)>0$ does {\bf not} imply that the $4$-variable power function 
$$
F_{p}(\omega)=F_{p,\R^{4}}(\omega)=|\omega|^{p},\quad \omega\in\R^{4}
$$
is $(A,B)$-convex in $\R^{4}\setminus\{0\}$. 

\begin{example} Consider the case when $A=aI_{\C^{d}}$ and $B=bI_{\C^{d}}$, $a,b>0$. Then by Lemma~\ref{l: N prop of Deltap}~(\ref{L4 (vi)}) we have $\Delta_{p}(A,B)>0$ for all $p>1$. By Lemma~\ref{l: N gen F4 bis}, on the other hand, $F_{p}$ is $(A,B)$-convex in $\R^{4}\setminus\{0\}$ if and only if 
$|1-2/p|\leq 2\sqrt{ab}/(a+b)$. The same result also follows from \cite[Proof of Th\'eor\`eme~7 and p. 19]{Bakry3}.
\end{example}

We show in Proposition~\ref{p: N 1} that the $4$-variable power function $F_{p}$ is $(A,B)$-convex in a subregion $S_{\kappa}$ depending on $A$, $B$ and $p$.

We also show in Proposition~\ref{p: N modpw} that in the complementary region, $\R^{4}\setminus S_{k}$, a suitable multiple of the sum of the $2$-variable power functions $G_{p}$ compensates for the lack of $(A,B)$-convexity of $F_{p}$ in $\R^{4}\setminus S_{k}$.
\medskip

Finally, we end up with the right sequence $\{\cP_{n}\}_{n\in\N_{+}}$:
$$
\cP_{n}(\zeta,\eta):=f_{n}(|\omega|)+K\left(f_{n}(|\zeta|)+f_{n}(|\eta|)\right),\quad \zeta,\eta\in\R^{2},\quad \omega=(\zeta,\eta)
$$
for suitable $\epsilon>0$ and $K>0$ depending on $p$, $A$ and $B$, and $f_{n}$ as in \eqref{eq: deffn}.

\section{The sequence $\{\cP_{n}\}_{n\in\N_{+}}$}\label{s: the sequence}
The aim of this section is to provide all the details in the construction of the sequence $\{\cP_{n}\}_{n\in\N_{+}}$ roughly described in Section~\ref{s: Thm bil discussion} and prove some of its properties. For the reader's convenience we also recollect here some notation from Sections~\ref{s: Neumann introduction} and \ref{s: Neumann gen conv}.

\subsection{Power functions in higher dimensions}\label{s: powers} 
Let $p>1$ and $l\in\N_{+}$. Define $F_p:(\R^{2})^{l}\rightarrow\R_+$ by
$$
F_p(\xi)=|\omega|^p,\quad \omega\in(\R^{2})^{l}.
$$
We remark that while the power functions defined above are different for different values of the dimension $l$, we will use the {\bf same} symbol ``$F_p$'' to denote all of them. 
\smallskip

For $\omega_{1},\dots,\omega_{l}\in \R^2$, $X_{1},\dots,X_{l}\in\R^{d}\times\R^{d}$ and $A_{1},\dots, A_{l}\in \cA(\Omega)$, we set 

$$
\omega=(\omega_1,\hdots,\omega_m),\quad X=(X_1,\hdots,X_m),\quad 
\bA=(A_1,\hdots,A_m).
$$
We also define
$$
\Delta_r(\bA)=\min_{j=1,\hdots,l}\Delta_r(A_j)\quad {\rm and}\quad 
\Lambda(\bA)=\max_{j=1,\hdots,l}\Lambda(A_j).
$$
\begin{lemma}\label{l: N gen F4 bis}
Let $p>1$, $d,l\in\N_{+}$, $A_{1},\dots,A_{l}\in\cA(\Omega)$, $\omega\in(\R^{2})^{l}$ and $X\in(\R^{d}\times\R^{d})^{l}$. Then, for $\omega\neq 0$,
\begin{equation}
\label{eq: gen Fl 1}
H_{F_p}^{\bA}[\omega;X]=|\omega|^{p-2}H_{F_p}^{\bA}[\omega/|\omega|;X].
\end{equation}
In case when $|\omega|=1$ we have, for
$$
\sigma_{j}:=e^{-i\arg\omega_j}\cV_d^{-1}(X_j)\in\C^{d},
$$
the following formul\ae\ :
\begin{enumerate}[(\rm I)]
\item
\label{eq: gen Fl 2}
$$
p^{-1}H_{F_p}^{\bA}[\omega;X]  =\sum_{j=1}^{l}\Re\sk{A_{j}\sigma_{j}}{\sigma_{j}}
+(p-2)\sum_{j,k=1}^{l}|\omega_{j}||\omega_k|\Re\sk{A_{j}\sigma_{j}}{\Re\sigma_{k}}
$$
\item
\label{eq: gen Fl 3}
$$
\aligned
p^{-1}H_{F_p}^{\bA}[\omega;X]  &=\sum_{j=1}^{l}\big(1-|\omega_j|^2\big)\Re\sk{A_{j}\sigma_{j}}{\sigma_{j}}
+\frac{p}{2}\sum_{j=1}^{l}|\omega_j|^2\Re\sk{A_{j}\sigma_{j}}{\cI_p\sigma_{j}}\\
&\hskip 40pt
+(p-2)\sum_{j\not =k}|\omega_{j}||\omega_k|\Re\sk{A_{j}\sigma_{j}}{\Re\sigma_{k}}.
\endaligned
$$
\end{enumerate}
\end{lemma}
\begin{proof}
A rapid calculation shows that
\begin{equation}\label{eq: formula Hessian}
(D^{2}F_{p})(\omega)=p|\omega|^{p-2}\left(I_{\R^{2l}}+(p-2)\frac{\omega}{|\omega|}\otimes\frac{\omega}{|\omega|}\right),\quad \omega\in(\R^{2})^{l}\setminus\{0\}
\end{equation}
and \eqref{eq: gen Fl 1} trivially follows from definitions. 

Now assume that $|\omega|=1$. From \eqref{eq: formula Hessian} 
we get 
$$
\aligned
p^{-1}H^{\bA}_{F_{p}}[\omega;X]
&=\sum_{j=1}^{m}
\sk{\left[\left(I_{\R^{2}}+|\omega_{j}|^{2}(p-2)\frac{\omega_{j}}{|\omega_{j}|}\otimes\frac{\omega_{j}}{|\omega_{j}|}\right)\otimes I_{\R^{d}}\right]X_j}{\cM(A_{j})X_{j}}\\
&\hskip 20pt+\sum_{j\not=k}|\omega_{j}||\omega_{k}|(p-2)\sk{\left[\left(\frac{\omega_{j}}{|\omega_{j}|}\otimes\frac{\omega_{k}}{|\omega_{k}|}\right)\otimes I_{\R^{d}}\right]X_k}{\cM(A_j)X_j}\\
&=:\Sigma_1+\Sigma_2.
\endaligned
$$
In order to calculate the $\Sigma_1$, write the summands as
$$
(1-|\omega_j|^2)\sk{\cM(A_{j})X_{j}}{X_{j}}
+|\omega_j|^2
\sk{\left[\left(I_{\R^{2}}+(p-2)\frac{\omega_{j}}{|\omega_{j}|}\otimes\frac{\omega_{j}}{|\omega_{j}|}\right)\otimes I_{\R^{d}}\right]X_j}{\cM(A_{j})X_{j}}.
$$
By applying \eqref{eq: real form} on the first term above and \eqref{eq: formula Hessian} and Lemma~\ref{l: N prop of Deltap}~(\ref{L4 (v)}) on the second, we get
\begin{equation}
\label{eq: S1}
\Sigma_1=\sum_{j=1}^{m}
\left(
(1-|\omega_j|^2)\Re\sk{A_{j}\sigma_{j}}{\sigma_{j}}
+\frac p2|\omega_j|^2
\Re\sk{A_{j}\sigma_{j}}{\cI_p\sigma_j}
\right).
\end{equation}
In order to calculate $\Sigma_2$, first write $\theta_{j}:=\arg\omega_{j}$, $j=1,\dots,l$. A calculation shows that
$$
\aligned
2\frac{\omega_{j}}{|\omega_{j}|}\otimes\frac{\omega_{k}}{|\omega_{k}|}&=\left(\begin{array}{rr}2\cos\theta_{j}\cos\theta_{k} & 2\cos\theta_{j}\sin\theta_{k} \\2\sin\theta_{j} \cos\theta_{k}& 2\sin\theta_{j}\sin\theta_{k}\end{array}\right)\\
&=\left(\begin{array}{rr}\cos(\theta_{j}-\theta_{k}) & -\sin(\theta_{j}-\theta_{k}) \\\sin(\theta_{j}-\theta_{k})& \cos(\theta_{j}-\theta_{k})\end{array}\right)\\
&\hskip 30pt
+\left(\begin{array}{rr}\cos(\theta_{j}+\theta_{k}) & -\sin(\theta_{j}+\theta_{k}) \\\sin(\theta_{j}+\theta_{k})& \cos(\theta_{j}+\theta_{k})\end{array}\right)\cdot \left(\begin{array}{rr}1 & 0 \\0 & -1\end{array}\right).
\endaligned
$$
Therefore,
$$
2\left(\frac{\omega_{j}}{|\omega_{j}|}\otimes\frac{\omega_{k}}{|\omega_{k}|}\right)\otimes I_{\R^{d}}=\cM\left(e^{i(\theta_{j}-\theta_{k})}I_{\C^{d}}\right)+ \cM\left(e^{i(\theta_{j}+\theta_{k})}I_{\C^{d}}\right)\left(\left(\begin{array}{cc}1 & 0 \\0 & -1\end{array}\right)\otimes I_{\R^{d}}\right).
$$
Observe that the last factor in the second term on the right-hand side is the real form of the complex conjugation in $\C^{d}$. Consequently, from the identity
\begin{equation}
\label{eq: X_s}
X_{k}=\cV_{d}\left(e^{i\theta_{k}}\sigma_{k}\right)
\end{equation}
we get
$$
\aligned
2\left[\left(\frac{\omega_{j}}{|\omega_{j}|}\otimes\frac{\omega_{k}}{|\omega_{k}|}\right)\otimes I_{\R^{d}}\right]X_{k}
&=\cV_{d}\left(e^{i(\theta_{j}-\theta_{k})}\cdot e^{i\theta_{k}}\sigma_{k}\right)
+\cV_{d}\left(e^{i(\theta_{j}+\theta_{k})}\cdot e^{-i\theta_{k}}\overline{\sigma_{k}}\right)\\
&=2\cV_{d}\left(e^{i\theta_j}\Re\sigma_{k}\right).
\endaligned
$$
By using \eqref{eq: X_s} with $k=j$ and \eqref{eq: real form} we conclude that
$$
\sk{\left[\left(\frac{\omega_{j}}{|\omega_{j}|}\otimes\frac{\omega_{k}}{|\omega_{k}|}\right)\otimes I_{\R^{d}}\right]X_k}{\cM(A_j)X_j}
=\Re\sk{\Re\sigma_{k}}{A_{j}\sigma_{j}}.
$$
Hence
\begin{equation}
\label{eq: S2}
\Sigma_2=(p-2)\sum_{j\not=k}|\omega_{j}||\omega_{k}|\Re\sk{\Re\sigma_{k}}{A_{j}\sigma_{j}}.
\end{equation}
The identity \eqref{eq: gen Fl 3} now follows by combining \eqref{eq: S1} with \eqref{eq: S2}.
\medskip

In order to prove \eqref{eq: gen Fl 2}, we write the diagonal terms in \eqref{eq: gen Fl 3} as
$$
\aligned
\frac p2|\omega_{j}|^{2} & \Re\sk{A_{j}\sigma_{j}}{\cI_{p}\sigma_j}+
\big(1-|\omega_{j}|^2\big)\Re\sk{A_{j}\sigma_{j}}{\sigma_{j}}\\
&=
\Re\sk{A_{j}\sigma_{j}}{\sigma_j}+|\omega_j|^{2}\Re\sk{A_{j}\sigma_{j}}{\frac p2\cI_p\sigma_{j}-\sigma_{j}},
\endaligned
$$
and use the identity
\begin{alignat*}{2}
\frac{p}{2}\cI_{p}\sigma_{j}-\sigma_{j}=(p-2)\Re\sigma_{j}.
\tag*{\qedhere}
\end{alignat*}
\end{proof}
We note that in the special case $l=1$, Lemma~\ref{l: N gen F4 bis} is consistent with Lemma~\ref{l: N prop of Deltap} \eqref{L4 (v)}.
\begin{corollary}
\label{c: Fl}
Let $p\geq 2$, $d,l\in\N_{+}$ and $A_{1},\dots,A_{l}\in\cA(\Omega)$. Then for every $\omega\in(\R^{2})^{l}$ with $|\omega|=1$ and for every $X\in(\R^{d}\times\R^{d})^{l}$ we have 
$$
p^{-1}H_{F_p}^{\bA}[\omega;X]  \geq
|X|^2\left(
\Delta_p(\bA)
-(p-2)\Lambda(\bA)\sum_{j<k}|\omega_{j}||\omega_k|
\right).
$$
\end{corollary}

\begin{proof}
For $j=1,\dots,l$ set
$$
\sigma_{j}:=e^{-i\arg\omega_j}\cV_d^{-1}(X_j)\in\C^{d}.
$$
From \eqref{eq: gen Fl 3} we get
$$
\aligned
p^{-1}H_{F_p}^{\bA}[\omega;X]  &\geq\sum_{j=1}^{l}\big(1-|\omega_j|^2\big)
|\sigma_j|^2\Delta_2(A_j)
+\frac{p}{2}\sum_{j=1}^{l}|\omega_j|^2|\sigma_j|^2\Delta_p(A_j)\\
&\hskip 40pt
-|p-2|\sum_{j\not =k}|\omega_{j}||\omega_k|
\Lambda(A_j)|\sigma_{j}||\sigma_{k}|.
\endaligned
$$
Since $\Delta_2(B)\geq\Delta_p(B)$ for every $p>1$ and every matrix $B$, we may continue as
$$
\aligned
p^{-1}H_{F_p}^{\bA}[\omega;X]  &\geq\Delta_p(\bA)\sum_{j=1}^{l}\big(1-|\omega_j|^2\big)
|\sigma_j|^2
+\frac{p}{2}\,\Delta_p(\bA)\sum_{j=1}^{l}|\omega_j|^2|\sigma_j|^2\\
&\hskip 40pt
-2|p-2|\Lambda(\bA)\sum_{j<k}|\omega_{j}||\omega_k|
|\sigma_{j}||\sigma_{k}|\\
&\geq\Delta_p(\bA)\min\left\{\frac p2,1\right\}|X|^2
-2|p-2|\Lambda(\bA)\sum_{j<k}|\omega_{j}||\omega_k|
|X_{j}||X_{k}|.
\endaligned
$$
Since $p\geq 2$, the corollary is proved.
\end{proof}
\subsection{Generalised convexity of the $4$-variable power function} For each $\kappa>0$ we consider the subregion of $\R^{4}$ given by
$$
S_{\kappa}:=\left\{(\zeta,\eta)\in\R^{2}\times\R^{2}: |\zeta|\leq\kappa^{-1}|\eta|\right\}\cup\left\{(\zeta,\eta)\in\R^{2}\times\R^{2}: |\eta|\leq\kappa^{-1}|\zeta|\right\}.
$$ 
Note that for $\kappa\in (0,1]$ we have $S_{\kappa}=\R^{4}$. Also,
when $\kappa>1$ and $\omega\in\R^{4}\setminus S_{\kappa}$, we have
\begin{equation}
\label{eq: N 1}
 \frac{1}{\kappa\sqrt{2}}|\omega|\leq |\zeta|\leq |\omega|\quad {\rm and}\quad \frac{1}{\kappa\sqrt{2}}|\omega|\leq |\eta|\leq |\omega|.
\end{equation}
If $p\geq 2$ and $A,B\in\cA(\Omega)$ are $p$-elliptic, then we define the constant
\begin{equation}\label{eq: def k_{p}}
\kappa_{p}
:=(p-2)\,\frac{\Lambda(A,B)}{\Delta_p(A,B)}.
\end{equation}
\begin{proposition}
\label{p: N 1}
Let $p>2$. Suppose that $A,B\in\cA(\Omega)$ satisfy $\Delta_{p}(A,B)>0$. Let $\kappa\geq \kappa_{p}$. Then $F_{p}$ is $(A,B)$-convex in the region $S_{\kappa}$, that is,
$$
H^{(A,B)}_{F_{p}}[\omega;(X,Y)]\geq 0,
$$
for all $\omega\in S_{\kappa}$ and for all $X,Y\in\R^{d}\times\R^{d}$.
\end{proposition}
\begin{proof}
The proposition follows from \eqref{eq: gen Fl 1} and Corollary~\ref{c: Fl}, applied with $l=2$, and the definition of $\kappa_{p}$.
\end{proof}
\subsection{Modified $4$-variable power function}\label{s: mod powers}
We now 
perturb the $4$-variable power function
$F_{p}$ in order to get a function $(A,B)$-convex in {\it all} of $\R^{4}$. Let $p>2$ and $A,B\in\cA(\Omega)$ such that $\Delta_{p}(A,B)>0$. Define $\kappa_{p}$ by \eqref{eq: def k_{p}} and set
\begin{equation}\label{eq: def K_{p}}
K_p:=
\begin{cases}
0; & \kappa_p\leq1\\
(2\kappa_p)^{p-1};& \kappa_p>1\,.
\end{cases}
\end{equation}
Consider the function
\begin{equation}\label{eq: Pp}
P_{p}(\zeta,\eta)
:=F_{p}(\zeta,\eta)+K_p\left(F_p(\zeta)+F_p(\eta)\right),\quad \zeta,\eta\in\R^{2}.
\end{equation}
\begin{proposition}
\label{p: N modpw}
Suppose that $\Delta_{p}(A,B)>0$. Then $P_{p}$
is $(A,B)$-convex in $\R^{4}$.
\end{proposition}
\begin{proof}
If $\kappa_{p}\leq 1$, then
$S_{\kappa_{p}}=\R^{4}$ and $P_{p}=F_{p}$ on $\R^{4}$. Hence the proposition in this case follows from Proposition~\ref{p: N 1}. 
\medskip

Suppose now that $\kappa_{p}>1$. Let $\omega=(\zeta,\eta)\in\R^{2}\times\R^{2}$ and $X,Y\in\R^{d}\times\R^{d}$. We have
\begin{equation}
\label{eq: latif}
H_{P_{p}}^{(A,B)}[\omega;(X,Y)]
=
H_{F_{p}}^{(A,B)}[\omega;(X,Y)]
+(2\kappa_p)^{p-1}\left(H_{F_p}^A[\zeta;X]+H_{F_p}^B[\eta;Y]\right).
\end{equation}
Since $\Delta_{p}(A,B)>0$, Proposition~\ref{p: N 1} and Lemma~\ref{l: N prop of Deltap}~(\ref{L4 (v)}) imply that $P_{p}$
is $(A,B)$-convex in the region $S_{\kappa_p}$. 

If $\omega\in\R^4\backslash{S_{\kappa_p}}$, we separately estimate the two terms in the right-hand side of \eqref{eq: latif}. 
Since $\Delta_p(A,B)>0$, Lemma \ref{l: N gen F4 bis} and Lemma~\ref{l: N prop of Deltap}~(\ref{L4 (v)}) give
\begin{equation}
\label{eq: ineq1}
H_{F_{p}}^{(A,B)}[\omega;(X,Y)]\geq 
-2p(p-2)|\omega|^{p-2}\Lambda(A,B)|X||Y|\,,
\end{equation}
while Lemma~\ref{l: N prop of Deltap}~(\ref{L4 (v)}) and \eqref{eq: N 1} give
\begin{equation}
\label{eq: ineq2}
H_{F_p}^A[\zeta;X]+H_{F_p}^B[\eta;Y]
\geq \frac{p^2\Delta_p(A,B)}{(\kappa_p\sqrt{2})^{p-2}}
|\omega|^{p-2}|X||Y|\,.
\end{equation}
In order to finish the proof, combine \eqref{eq: latif}, \eqref{eq: ineq1} and \eqref{eq: ineq2}. 
\end{proof}
\subsection{Definition of $\cP_{n}$}\label{s: def Pn}
Fix $p>2$ and $A,B\in\cA(\Omega)$ such that $\Delta_{p}(A,B)>0$. By Lemma~\ref{l: N prop of Deltap}~(\ref{L4 (iii)}) there exists $\epsilon=\epsilon(p,A,B)>0$ such that $\Delta_{p+\epsilon}(A,B)>0$. For this particular $\epsilon>0$ and all $n\in\N_{+}$ define $f_{n}$ by \eqref{eq: deffn}. For every $l\in\N_{+}$ define $\cF_{n}:(\R^{2})^{l}\rightarrow\R_{+}$ by
$$
\cF_{n}(\omega):=f_{n}(|\omega|),\quad \omega\in(\R^{2})^{l}.
$$
Let $\kappa_{p+\epsilon}$ and $K_{p+\epsilon}$ be the two constants given by \eqref{eq: def k_{p}} and \eqref{eq: def K_{p}}. We define
$$
\cP_n(\zeta,\eta)
:=\cF_{n}(\zeta,\eta)+K_{p+\epsilon}\left(\cF_{n}(\zeta)+\cF_{n}(\eta)\right),\quad (\zeta,\eta)\in\R^{2}\times\R^{2}.
$$
For any $n\in\N_{+}$, consider the set $\Theta_{n}\subset\R^4$ defined by
$$
\Theta_{n}=\left\{|\zeta|^{2}+|\eta|^{2}=n^{2}\right\}\cup\left\{|\zeta|=n\right\}\cup\left\{|\eta|=n\right\}.
$$
\begin{proposition}
\label{p: N 2}
\ 
\begin{enumerate}[\rm (i)]  
\item 
\label{eq: y1}
$\cP_{n}\in C^{1}(\R^{4})\cap C^{2}(\R^{4}\setminus\Theta_{n})$ for all $n\in\N_{+}$. Moreover,
$$
\aligned
D\cP_{n}\rightarrow 0 
&\quad  \text{pointwise in  } \R^{4}\\
D^{2}\cP_{n}\rightarrow 0 &
\quad\text{pointwise in  }\R^{4}\setminus\bigcup_{k\in\N_{+}}\Theta_{k}
\endaligned
$$
as $n\rightarrow\infty$. 
\item 
\label{eq: y2}
$\cP_{n}$ is $(A,B)$-convex in $\R^{4}\setminus\Theta_{n}$, for all $n\in\N_{+}$. Moreover, for 
all $n\in\N_{+}$ and all $\omega\in \R^{4}\setminus\Theta_{n}$ with $|\omega|>n$, we have
$$
H^{(A,B)}_{\cP_{n}}[\omega;(X,Y)]\geq (p+\epsilon)n^{p-2}\lambda(A,B)\left(|X|^{2}+|Y|^{2}\right),\quad \forall X,Y\in\R^{2d}.
$$
\item 
\label{eq: y3}
There exists $C>0$ that does not depend on $n$ such that
$$
\aligned
\mod{(D \cP_{n})(\omega)}&\leq C\left(|\zeta|^{p-1}+|\eta|^{p-1}\right),\quad \forall\omega\in\R^{4},\quad \forall n\in\N_{+};\\
\mod{(D^{2}\cP_{n})(\omega)}&\leq C\left(|\zeta|^{p-2}+|\eta|^{p-2}\right),\quad \forall\omega\in\R^{4}\setminus\Theta_{n},\quad \forall n\in\N_{+}.
\endaligned
$$
\item
\label{eq: y4}
For every $n\in\N_{+}$ there exists $C(n)>0$ such that
$$
\mod{D\cP_{n}(\omega)}\leq C(n)|\omega|,\quad \forall\omega\in\R^{4}.
$$
\item
\label{eq: y5}
$D^{2}\cP_{n}\in L^{\infty}(\R^{4}\setminus\Theta_{n};\R^{4\times4})$, for all $n\in\N_{+}$.
\end{enumerate}
 \end{proposition}
\begin{proof}
Item (\ref{eq: y1}) is an immediate consequence of the definition of $\cP_n$.

We now prove item (\ref{eq: y2}). Let $\omega=(\zeta,\eta)\in(\R^{2}\times\R^{2})\setminus\Theta_{n}$. Suppose first that $|\omega|<n$; then $|\zeta|<n$ and $|\eta|<n$. Hence, in this case, $
\cF_n(\zeta)=n^{-\epsilon}F_{p+\epsilon}(\zeta)$, $\cF_n(\eta)=n^{-\epsilon}F_{p+\epsilon}(\eta)$ and $\cF_{n}(\omega)=n^{-\epsilon}F_{p+\epsilon}(\omega)$.
Therefore $\cP_n(\omega)=n^{-\epsilon}P_{p+\epsilon}(\omega)$ for all $\omega\in \R^{4}\setminus\Theta_{n}$ with $|\omega|<n$ and the $(A,B)$-convexity follows from Proposition~\ref{p: N modpw}. 

Suppose now that
$|\omega|>n$. Then,
$$
\cP_{n}(\omega)=\frac{p+\epsilon}{2}n^{p-2}|\omega|^{2}+\left(1-\frac{p+\epsilon}{2}\right)n^{p}+K_{p+\epsilon}\left(\cF_n(\zeta)+\cF_n(\eta)\right).
$$
Therefore,
$$
\aligned
H^{(A,B)}_{\cP_{n}}&[\omega;\cW_{2,d}(\sigma_{1},\sigma_{2})]\\
&=a\left(\Re\sk{A\sigma_{1}}{\sigma_{1}}+\Re\sk{B\sigma_{2}}{\sigma_{2}}\right)+K_{p+\epsilon}\left(H^{A}_{\cF_{n}}[\zeta;\cV_{d}(\sigma_{1})]+H^{B}_{\cF_{n}}[\eta;\cV_{d}(\sigma_{2})]\right)\\
&\geq a\lambda(A,B)\left(|\sigma_{1}|^{2}+|\sigma_{2}|^{2}\right)+K_{p+\epsilon}\left(H^{A}_{\cF_{n}}[\zeta;\cV_{d}(\sigma_{1})]+H^{B}_{\cF_{n}}[\eta;\cV_{d}(\sigma_{2})]\right),
\endaligned
$$
where $a=(p+\epsilon)n^{p-2}$. Since 
$$
H^{C}_{\cF_{n}}[u;\xi]=
\left\{
\begin{array}{rl}
{\displaystyle n^{-\epsilon}H^{C}_{F_{p+\epsilon}}[u;\xi],}& |u|<n;\\
{\displaystyle \tfrac{p+\epsilon}{2}n^{p-2}H^{C}_{F_{2}}[u;\xi],}&|u|>n
\end{array}
\right.
$$
and $\Delta_{p+\epsilon}(A,B)>0$, we deduce from Lemma~\ref{l: N prop of Deltap} that 
$$
H^{(A,B)}_{\cP_{n}}[\omega;\cW_{2,d}(\sigma_{1},\sigma_{2})]\geq   (p+\epsilon)n^{p-2}\lambda(A,B)\left(|\sigma_{1}|^{2}+|\sigma_{2}|^{2}\right)
$$
for all $\sigma_{1},\sigma_{2}\in\C^{d}$ and all $\omega\in \R^{4}\setminus\Theta_{n}$ with $|\omega|>n$. This finishes the proof of item (\ref{eq: y2}). 
Items (\ref{eq: y3}), (\ref{eq: y4}) and (\ref{eq: y5})  easily follow from definitions.
\end{proof}

\section{The sequence $\{\cR_{n,\nu}\}_{n\in\N_{+}}$}\label{s: conv}
Let $p> 2$ and $q=p/(p-1)$. Fix $A,B\in\cA(\Omega)$ with $\Delta_{p}(A,B)>0$. Let $\cQ=\cQ_{\delta}$ denote the Nazarov-Treil Bellman function introduced in \eqref{eq: N Bellman} with $\delta>0$ chosen so that Theorem~\ref{t: N B gen conv} holds true.  

Fix a radial function $\f\in C_c^\infty(\R^4)$ such that $0\leq\f\leq1$, ${\rm supp}\, \f\subset B_{\R^{4}}(0,1)$ and $\int\f=1$. Also, fix a radial function $\psi\in C^{\infty}_{c}(\R^{4})$ such that $\psi\geq 0$, $\psi=1$ on $\{|\omega|\leq 3\}$ and $\psi=0$ on $\{|\omega|>4\}$. For $\nu\in(0,1]$ and $n\in\N_{+}$ define $\varphi_\nu(\omega)=\nu^{-4}\varphi(\omega/\nu)$ and $\psi_{n}(\omega)=\psi(\omega/n)$. 
\subsection*{Notation}
Let $\{\cP_{n}\}_{\in\N_{+}}$ be the sequence of Section~\ref{s: def Pn}. For every $n\in\N_{+}$ and all $\nu\in (0,1]$, define
\begin{equation}\label{eq: d Rnv}
\aligned
&\cQ_{n,\nu}:=\psi_{n}\cdot (\cQ\star\varphi_{\nu});\\
&\cR_{n,\nu}:=\cQ_{n,\nu}+C_{1}\nu^{q-2}(\cP_{n}\star\varphi_{\nu}),
\endaligned
\end{equation}
where $C_{1}=C_{1}(p,A,B,\psi)>0$ is a constant not depending on $\nu$ which will be fixed later.
\subsection{Estimates for $\cQ\star\varphi_{\nu}$}\label{s: estest} 
Next result was proven in \cite[Corollary 5.5]{CD-DivForm}.

\begin{proposition}\label{p: N gen conv reg} 
Suppose that $p\geq2$ and $A,B\in \cA(\Omega)$ satisfy $\Delta_p(A,B)>0$. Then $\cQ\star\varphi_{\nu}$ is $(A,B)$-convex in $\R^{4}$.
More specifically, for almost every $x\in\Omega$ we have
$$
H_{\cQ\star\varphi_{\nu}}^{(A(x),B(x))}[\omega;(X,Y)]
\geq \frac{\Delta_p(A,B)}{5}\cdot\frac{\lambda(A,B)}{\Lambda(A,B)}|X||Y|,
$$
for any $\omega\in\R^{4}$ and $X,Y\in\R^{2d}$.
\end{proposition}
We shall need estimates of the first- and second-order partial derivatives of $Q*\varphi_{\nu}$. As a consequence of \eqref{eq: N 5} we have (recall that $\delta$ is fixed): 
\begin{equation}
\label{eq: N 6}
\aligned
0\leqslant (\cQ\star\varphi_{\nu})(\zeta,\eta)
&\leqsim_{p} (|\zeta|+\nu)^p+(|\eta|+\nu)^q,\\
\left|\partial_\zeta (\cQ\star\varphi_{\nu})(\zeta,\eta)\right|
&\leqsim_{p}\max\left\{(|\zeta|+\nu)^{p-1}, |\eta|+\nu\right\},\\ 
\left|\partial_\eta (\cQ\star\varphi_{\nu})(\zeta,\eta)\right|
&\leqsim_{p} (|\eta|+\nu)^{q-1},
\endaligned
\end{equation}
for all $\zeta,\eta\in\R^{2}$ and $\nu\in(0,1)$, see \cite[Theorem~4]{CD-Riesz}. Also, a calculation shows that
\begin{equation}\label{eq: Neumann 7}
\mod{(D^{2}\cQ)(\zeta,\eta)}\leqsim_{p,\delta}\, |\zeta|^{p-2}+|\eta|^{q-2}+|\eta|^{2-q}+1,
\end{equation}
for all $(\zeta,\eta)\in\R^{4}\setminus\Upsilon$, where $\Upsilon$ is defined on page \pageref{eq: N flow}.
\begin{lemma}
\label{l: N second order}
There exists $C=C(p,\delta)>0$ such that 
\begin{enumerate}[{\rm (i)}]
\item
\label{eq: est Q*phi}
${\displaystyle \hskip 5pt|(\cQ\star\varphi_{\nu})(\omega)|\leq C\left(|\omega|^{p}+|\omega|^{q}+1\right);}$
\item
\label{eq: est Q*phi i)}
${\displaystyle \hskip 5pt\mod{D(\cQ\star\varphi_{\nu})(\omega)}\leq C\left(|\omega|^{p-1}+|\omega|^{q-1}\right);}$
\item
\label{eq: est Q*phi ii)}
${\displaystyle \mod{D^{2}(\cQ\star\varphi_{\nu})(\omega)}
\leq C\nu^{q-2}\left( |\zeta|^{p-2}+|\eta|^{2-q}+1\right)}$
\end{enumerate}
for all $\nu\in (0,1]$ and $\omega=(\zeta,\eta)\in\R^{2}\times\R^{2}$. 
\end{lemma}
\begin{proof}
Item (\ref{eq: est Q*phi}) directly follows from the first estimate in \eqref{eq: N 6}. Item (\ref{eq: est Q*phi ii)}) follows from \eqref{eq: Neumann 7} and the properties of convolution. Let us only treat in detail the convolution with the term with the negative exponent, $|\eta|^{q-2}$. We have
$$
\aligned
\int_{\R^{4}}|\eta^{\prime}|^{q-2}
&\varphi_{\nu}(\zeta-\zeta^{\prime},\eta-\eta^{\prime})\wrt\zeta^{\prime}\wrt\eta^{\prime}\\
&=\nu^{-2}\int_{\R^{2}}|\eta^{\prime}|^{q-2}\left[\nu^{-2}\int_{\R^{2}}\varphi\left(\frac{\zeta-\zeta^{\prime}}{\nu},\frac{\eta-\eta^{\prime}}{\nu}\right)\wrt\zeta^{\prime}\right]\wrt\eta^{\prime}\\
&=\nu^{-2}\int_{\R^{2}}|\eta^{\prime}|^{q-2}\int_{\R^{2}}\varphi\left(\zeta^{\prime},\frac{\eta-\eta^{\prime}}{\nu}\right)\wrt\zeta^{\prime}\wrt\eta^{\prime}\\
&=\nu^{-2}\int_{\{|\eta^{\prime}-\eta|<\nu\}}|\eta^{\prime}|^{q-2}\int_{\{|\zeta^{\prime}|<1\}}\varphi\left(\zeta^{\prime},\frac{\eta-\eta^{\prime}}{\nu}\right)\wrt\zeta^{\prime}\wrt\eta^{\prime}\\
&\leq \norm{\varphi}{\infty}|B_{\R^{2}}(0,1)|\nu^{-2}\int_{\{|\eta^{\prime}-\eta|\leq\nu\}}|\eta^{\prime}|^{q-2}\wrt\eta^{\prime}\\
&\leqsim\nu^{-2}
\left(\int_{\{|\eta^{\prime}-\eta|\leq\nu\}\cap\{|\eta^{\prime}|\geq\nu\}}+\int_{\{|\eta^{\prime}|<\nu\}}\right)|\eta^{\prime}|^{q-2}\wrt\eta^{\prime}\\
&\leqsim\nu^{-2}
\left(\int_{\{|\eta^{\prime}-\eta|\leq\nu\}}\nu^{q-2}\wrt\eta^{\prime}+\int_0^\nu r^{q-2} r\wrt r\right)\\
&\leqsim \nu^{q-2}.
\endaligned
$$
Now we prove (\ref{eq: est Q*phi i)}).
Let $j\in \{1,2,3,4\}$. Since $\cQ$ and $\varphi_{\nu}$ are even functions in each of the variables $\zeta_{1},\zeta_{2},\eta_{1},\eta_{2}$, function $\cQ\star\varphi_{\nu}$ also has this property, so 
\begin{equation}
\label{eq: ENG-SWE}
D_{j}(\cQ\star\varphi_{\nu})(0,0)=0.
\end{equation}

\noindent Hence, by item (\ref{eq: est Q*phi ii)}) and the mean value theorem, if $|\omega|<\nu\leq 1$ we get
$$
\mod{D_{j}(\cQ\star\varphi_{\nu})(\omega)}
\leq\underset{|\omega|\leq 1}{\max}\mod{D^{2}(\cQ\star\varphi_{\nu})(\omega)}|\omega|\leq C\nu^{q-2}|\omega|
\leq C|\omega|^{q-1}.
$$
By the second and third estimate in \eqref{eq: N 6}, there exists $C>0$ not depending on $\nu\in (0,1)$ and such that
\begin{alignat*}{2}
\mod{D_{j}(\cQ\star\varphi_{\nu})(\omega)}\leq C\left(|\omega|^{p-1}+|\omega|^{q-1}\right),\quad \forall |\omega|\geq\nu,\quad \forall\nu\in (0,1].
\tag*{\qedhere}
\end{alignat*}
\end{proof}
\subsection{Estimates for $\cP_{n}\star\varphi_{\nu}$}\label{s: Pn*phi}
Since $ \cP_{n}\in C^1(\R^4)$ and its second-order partial derivatives exist on $\R^4\setminus \Theta_{n}$ and extend to a locally integrable function on $\R^4$, we have
\begin{equation}
\label{eq: N conv}
D_{j}(\cP_{n}\star\varphi_{\nu})=(D_{j}\cP_{n})\star\varphi_{\nu};\quad D^{2}_{ij}(\cP_{n}\star\varphi_{\nu})=(D^{2}_{ij}\cP_{n})\star\varphi_{\nu},\quad i,j=1,\dots,4.
\end{equation}

\begin{proposition}\label{t: N Pn conv}
Let $\nu\in (0,1)$.
\begin{enumerate}[\rm (i)]
\item 
\label{eq: alons}
The functions $D(\cP_{n}\star\varphi_{\nu})$ and $D^2(\cP_{n}\star\varphi_{\nu})$ converge pointwise to $0$ in $\R^{4}$ as $n\rightarrow\infty$.
\item
\label{eq: enfants}
The function $\cP_{n}\star\varphi_{\nu}$ is $(A,B)$-convex in $\R^{4}$. Moreover, for all $n\in\N_{+}$, $X,Y\in\R^{2d}$ and all $\omega$ with $|\omega|>2n$, 
$$
H^{(A,B)}_{\cP_{n}\star\varphi_{\nu}}[\omega;(X,Y)]
\geq 
(p+\epsilon)n^{p-2}\lambda(A,B)\left(|X|^{2}+|Y|^{2}\right).
$$
\item
\label{eq: de}
There exists $C>0$ that does not depend on $n$ and $\nu$ such that
$$
\mod{D(\cP_{n}\star\varphi_{\nu})(\omega)}\leq C(|\omega|^{p-1}+|\omega|^{q-1}),\quad \forall\omega\in\R^{4},\quad\forall n\in\N_{+}.
$$
\item
\label{eq: la}
For every $n\in\N_{+}$ there exists $C(n)>0$ (that does not depend on $\nu$) such that
$$
\mod{D(\cP_{n}*\varphi_{\nu})(\omega)}\leq C(n)|\omega|,\quad \forall\omega\in\R^{4}.
$$
\item 
\label{eq: patrie}
$D^{2}(\cP_{n}\star\varphi_{\nu})\in L^{\infty}(\R^{4};\R^{4\times4})$ 
and $\norm{D^{2}(\cP_{n}\star\varphi_{\nu})}{\infty}\leq C(n)$ independently of $\nu$.
\end{enumerate}
\end{proposition}
\begin{proof}
Item~(\ref{eq: alons}) follows by combining \eqref{eq: N conv}, Proposition~\ref{p: N 2} (\ref{eq: y1}) and (\ref{eq: y3}) with the dominated convergence theorem. Item~(\ref{eq: patrie}) follows from \eqref{eq: N conv} and Proposition~\ref{p: N 2} (\ref{eq: y5}).

By \eqref{eq: N conv} we have
$$
\aligned
H_{ \cP_{n}\star\varphi_{\nu}}^{(A(x),B(x))}[\omega;(X,Y)]&=\int_{\R^4}H_{ \cP_{n}}^{(A(x),B(x))}[\omega-\omega^{\prime};(X,Y)]\varphi_\nu(\omega^{\prime})\,\wrt\omega^{\prime},
\endaligned
$$
for all $x\in\Omega$, $\omega\in\R^{4}$ and $X,Y\in\R^{2d}$.
Since we assumed that $|\omega|>2n$ and since the support of the integrand is contained in $B_{\R^{4}}(0,\nu)$, we have $|\omega-\omega'|>2n-\nu>n$, therefore we may estimate the integrand by means of Proposition~\ref{p: N 2}~(\ref{eq: y2}) almost everywhere on $B_{\R^{4}}(0,\nu)$ and thus prove item~(\ref{eq: enfants}).
\medskip

Let us address item (\ref{eq: de}). We proceed much as in the proof of Lemma~\ref{l: N second order}~(\ref{eq: est Q*phi i)}). First consider $|\omega|\leq1$. 
The function $\cP_{n}\star\varphi_{\nu}$ is smooth and even in $\zeta_{1},\zeta_{2},\eta_{1},\eta_{2}$, so 
\begin{equation}
\label{eq: massena}
D(\cP_{n}\star\varphi_{\nu})(0)=0. 
\end{equation}
Hence, the second identity in \eqref{eq: N conv}, the second estimate of Proposition~\ref{p: N 2}~(\ref{eq: y3}) and the mean value theorem imply
$$
\mod{D(\cP_{n}*\varphi_{\nu})(\omega)}\leq C|\omega|\leq C|\omega|^{q-1},\quad \forall |\omega|\leq 1,\quad \forall n\in\N_{+}.
$$
Now take $|\omega|>1$. 
From the first identity in \eqref{eq: N conv} and the first estimate of Proposition~\ref{p: N 2}~(\ref{eq: y3}) we get
\begin{equation*}
\mod{D (\cP_{n}\star\varphi_{\nu})(\omega)}\leq C|\omega|^{p-1}.
\end{equation*} 
Thus we proved (\ref{eq: de}).

Finally, item (\ref{eq: la}) follows from item (\ref{eq: patrie}), \eqref{eq: massena} and the  mean value theorem.
\end{proof}

\subsection{Estimates for $\cR_{n,\nu}$} Recall the definition of $\cQ_{n,\nu}$ and $\cR_{n,\nu}$ in \eqref{eq: d Rnv}. It follows from Lemma~\ref{l: N second order} that there exists $C_{0}=C_{0}(p,\psi)>0$ such that 
\begin{equation}
\label{eq: N annulus}
\aligned
\mod{(D^{2}\cQ_{n,\nu})(\omega)}\leq C_{0} \nu^{q-2}n^{p-2},
\endaligned
\end{equation}
for every $\omega\in\R^{4}$ with $|\omega|\leq 5n$, and all $n\in\N_{+}$ and $\nu\in (0,1]$.
\begin{theorem}\label{t: N approx}
Let $\nu\in (0,1]$. There exists $C_{1}=C_{1}(p,A,B,\psi)>0$, not depending on $\nu$, such that $\cR_{n,\nu}$ is $(A,B)$-convex in $\R^{4}$ for all $n\in\N_{+}$. Moreover, the following statements hold.
\begin{enumerate}[{\rm (i)}]
\item
\label{eq: R1}
$D^{2}\cR_{n,\nu}\in L^{\infty}(\R^{4};\R^{4\times4})$.
\item
\label{eq: R2} 
We have 
$$
\aligned
D\cR_{n,\nu}  &\rightarrow D(\cQ\star\varphi_{\nu}),\\
D^2\cR_{n,\nu} &\rightarrow D^2(\cQ\star\varphi_{\nu})
\endaligned
$$
pointwise in $\R^4$ as $n\rightarrow\infty$.
\item
\label{eq: R3}
For any $n\in\N_{+}$ there exists $C(n,\nu,C_1)>0$ such that
$$
\mod{(D\cR_{n,\nu})(\omega)}\leq C(n,\nu)|\omega|,\quad \forall\omega\in\R^{4}.
$$
\item
\label{eq: R4}
There exists $C=C(\nu)>0$ that does not depend on $n$ such that
$$
\mod{(D\cR_{n,\nu})(\omega)}\leq C(\nu)\left(|\omega|^{p-1}+|\omega|^{q-1}
\right),
$$
for all $\omega\in\R^{4}$, $n\in\N_{+}$ and $\nu\in (0,1]$.
\item
\label{eq: R5} For any $n\in\N_{+}$ and $\nu>0$ we have
$$
(\partial_{\zeta} \cR_{n,\nu})(0,\eta)=0,\quad (\partial_{\eta} \cR_{n,\nu})(\zeta,0)=0,
$$
for all $\zeta,\eta\in\R^{2}$.
\end{enumerate}
\end{theorem}
\begin{proof}
The $(A,B)$-convexity in the region $\{|\omega|< 3n\}\cup\{|\omega|>4n\}$ follows, for any $C_1>0$, from the $(A,B)$-convexity of $\cQ\star\varphi_{\nu}$ and $\cP_{n}*\varphi_{\nu}$; see Proposition~\ref{p: N gen conv reg} and the first part of Proposition~\ref{t: N Pn conv}~(\ref{eq: enfants}). In order to achieve $(A,B)$-convexity in the region $\{3n\leq|\omega|\leq 4n\}$, we choose $C_1$ large enough and combine \eqref{eq: N annulus} with the second part of Proposition~\ref{t: N Pn conv}~(\ref{eq: enfants}).

Item~(\ref{eq: R1}) follows from Proposition~\ref{t: N Pn conv} (\ref{eq: patrie}) and the fact that $\cQ_{n,\nu}\in C^{2}_c(\R^{4})$ (or from \eqref{eq: N annulus}).

Item~(\ref{eq: R2}) is a trivial consequence of Proposition~\ref{t: N Pn conv} (\ref{eq: alons}) and the definition of $\cQ_{n,\nu}$. 

From \eqref{eq: ENG-SWE} and the fact that $\psi_n\equiv1$ in a neighbourhood of $0$, we conclude that $(D\cQ_{n,\nu})(0)=0$. Hence, by the mean value theorem and the fact that $\cQ_{n,\nu}\in C_c^\infty(\R^{4})$, we get $\mod{(D\cQ_{n,\nu})(\omega)}\leq C(\nu,n)|\omega|$. Item~(\ref{eq: R3}) follows from here and Proposition~\ref{t: N Pn conv}~(\ref{eq: la}).
\medskip

Item~(\ref{eq: R4}) follows by combining Lemma~\ref{l: N second order}~(\ref{eq: est Q*phi}) and (\ref{eq: est Q*phi i)}) with Proposition~\ref{t: N Pn conv}~(\ref{eq: de}).
In particular, use the fact that $D\psi_n\equiv0$ on $\{\omega: |\omega|\not\in[3n,4n]\}$, while, by Lemma~\ref{l: N second order}~(\ref{eq: est Q*phi}), on $\{\omega:|\omega|\in[3n,4n]\}$ we have the estimate 
$$
\aligned
|(D\psi_n)(\omega)|\cdot|(\cQ\star\varphi_\nu)(\omega)|
&\leqsim \frac{C(p,\delta)}{n}\cdot \left(1+|\omega|^p+|\omega|^q\right)\\
&\leqsim C(p,\delta)\left(1+|\omega|^{p-1}+|\omega|^{q-1}\right).
\endaligned
$$
Finally, $1\leq |\omega|^{p-1}+|\omega|^{q-1}$, because $|\omega|>1$.

To prove item~(\ref{eq: R5}) just observe that $\cR_{n,\nu}$ is smooth and even in each of the variables $\zeta_{1},\zeta_{2},\eta_{1}$ and $\eta_{2}$, because both $\cQ\star\varphi_{\nu}$ and $\cP_{n}\star\varphi_{\nu}$ have this property.
\end{proof}

\section{Proof of the bilinear embedding (Theorem~\ref{t: N bil})}
\label{s: N proof bil}

As we annunced in Sections~\ref{s: Neumann gen conv}~and~\ref{s: Thm bil discussion}, to prove Theorem~\ref{t: N bil} we modify the heat-flow-Bellman method of \cite{CD-DivForm} by means of the sequence $\{\cR_{n,\nu}\}$ of Theorem~\ref{t: N approx}.

Let $\Omega\subseteq\R^{d}$ be open. Fix two closed subspaces $\oV$ and $\oV^{\prime}$ of $W^{1,2}(\Omega)$ of the type discussed in Section~\ref{s: boundary}. Instead of proving \eqref{eq: N bil} directly, it is more convenient to show that 
\begin{equation}\label{eq: N bil bis}
\Delta_{p}(A,B)>0
\hskip 5pt
\Longrightarrow 
\hskip 5pt
\int^{\infty}_{0}\!\int_{\Omega}\mod{\nabla T^{A,\oV}_{t}f(x)}\mod{\nabla T^{B,\oV^{\prime}}_{t}g(x)}\wrt x\wrt t\leqsim \left(\|f\|^{p}_{p}+\|g\|^{q}_{q}\right),
\end{equation}
for all $f,g\in (L^{p}\cap L^{q})(\Omega)$.
Once \eqref{eq: N bil bis} is proved, \eqref{eq: N bil} follows by replacing $f$ and $g$ in \eqref{eq: N bil bis} with $sf$ and $s^{-1}g$ and minimising the right-hand side with respect to $s>0$.
\medskip

We first discuss analyticity of the semigroups in \eqref{eq: N bil}. Recall the notation $q=p/(p-1)$.
\begin{lemma}\label{l: N analytic sem}
Let $p\geq 2$ and $A\in\cA(\Omega)$. Suppose that $\Delta_{p}(A)>0$. Then there exists $\theta_{0}=\theta_{0}(p)\in (0,\pi/2)$ such that $(T^{A,\oV}_{t})_{t>0}$ is analytic and contractive in $L^{r}(\Omega)$ in the cone 
$$
\bS_{\theta_{0}}=\{z\in\C\setminus\{0\}: |\arg(z)|<\theta_{0}\},
$$
for all $r\in [q,p]$.
\end{lemma}
\begin{proof}
By complex interpolation it would be sufficient to prove the statement for $r=p,q$, but we prefer to avoid interpolation and prove the lemma directly for all $r$.

By Lemma~\ref{l: N prop of Deltap}~(\ref{L4 (iv)}), (\ref{L4 (i)}) and (\ref{L4 (iii)}) there exists $\theta_{0}=\theta_{0}(p)>0$ such that $\Delta_{r}(e^{i\theta} A)>0$ for all $\theta\in [-\theta_{0},\theta_{0}]$ and all $r\in [q,p]$. The contractivity now follows from Proposition~\ref{p: N Nittka} and the relation $T^{e^{i\theta}A,\oV}_{t}=T^{A,\oV}_{e^{i\theta}t}$. Finally, analyticity is a consequence of a standard argument \cite[Chapter II, Theorem 4.6]{EN}. 
\end{proof}
\begin{remark}\label{r: N analytic sem}
In the statement of Lemma~\ref{l: N analytic sem} we can take any $\theta_{0}$ with $\Delta_{p}(e^{\pm i \theta_{0}}A)>0$.
\end{remark}
For proving \eqref{eq: N bil bis} we also need the following result that should be compared with \cite[Lemma~4]{Egert2018}. Note that here the chain-rule is not a problem, because $\cR_{n,\nu}$ is smooth. 
\begin{lemma}\label{l: invariance}
Let $u\in\oV$ and $v\in\oV^{\prime}$. Then
$$
(\partial_{\zeta}\cR_{n,\nu})(u,v)\in\oV\quad \text{and}\quad (\partial_{\eta}\cR_{n,\nu})(u,v)\in\oV^{\prime},
$$
for all $n\in\N_{+}$ and $\nu>0$.
\end{lemma}
\begin{proof}
We prove the lemma in the case when $\oV=W^{1,2}_{D}(\Omega)$ and $\oV^{\prime}=W^{1,2}_{D^{\prime}}(\Omega)$, for two closed subsets $D,D^{\prime}\subseteq\partial\Omega$. The other cases are simpler and will not be written down here. 

Define $\varphi:=(\partial_{\zeta}\cR_{n,\nu})(u,v)$ and $\psi:=(\partial_{\eta}\cR_{n,\nu})(u,v)$. Let $u_{\ell}\in C^{\infty}_{c}(\R^{d}\setminus D)$ and $v_{\ell}\in C^{\infty}_{c}(\R^{d}\setminus D^{\prime})$ be such that ${u_{\ell}}_{\vert_{\Omega}}\rightarrow u$ and ${v_{\ell}}_{\vert_{\Omega}}\rightarrow v$ in $W^{1,2}(\Omega)$ as $\ell\rightarrow \infty$. Set $\varphi_{\ell}:=(\partial_{\zeta}\cR_{n,\nu})(u_{\ell},v_{\ell})$ and $\psi_{\ell}:=(\partial_{\eta}\cR_{n,\nu})(u_{\ell},v_{\ell})$. By Theorem~\ref{t: N approx}~(\ref{eq: R5}) we have ${\rm supp}(\varphi_{\ell})\subseteq  \R^{d}\setminus D$ and ${\rm supp}(\psi_{\ell})
\subseteq \R^{d}\setminus D^{\prime}$ so, since $\cR_{n,\nu}$ is smooth, we have $\varphi_{\ell}\in C^{\infty}_{c}(\Omega\setminus D)$ and $\psi_{\ell}\in C^{\infty}_{c}(\Omega\setminus D^{\prime})$. To conclude the proof we now proceed much as in \cite[Lemma~4]{Egert2018}, but with the simplification that here we can use the chain-rule for the composition of smooth functions. It follows from Theorem~\ref{t: N approx}~(\ref{eq: R1}) and the mean value theorem that
$$
\norm{\varphi-\varphi_{\ell}}{2}+\norm{\psi-\psi_{\ell}}{2}\leq C(n,\nu)\left(\norm{u-u_{\ell}}{2}+\norm{v-v_{\ell}}{2}\right).
$$
Therefore $\varphi_{\ell}\rightarrow \varphi$ and $\psi_{\ell}\rightarrow\psi$ in $L^{2}(\Omega)$. Also, by the chain-rule and Theorem~\ref{t: N approx}~(\ref{eq: R1}), the sequence $(\varphi_{\ell})_{\ell\in\N_{+}}$ is bounded in $W^{1,2}_{D}(\Omega)$ and the sequence $(\psi_{\ell})_{\ell\in\N_{+}}$ is bounded in $W^{1,2}_{D^{\prime}}(\Omega)$. Hence they admit two subsequences weakly convergent in $W^{1,2}_{D}(\Omega)$ and $W^{1,2}_{D^{\prime}}(\Omega)$, respectively. It follows that $\varphi\in W^{1,2}_{D}(\Omega)$ and $\psi\in W^{1,2}_{D^{\prime}}(\Omega)$.
\end{proof}
\subsection{Proof of \eqref{eq: N bil bis}} 
By Lemma~\ref{l: N prop of Deltap}~(\ref{L4 (i)}), we have $\Delta_{p}(A,B)=\Delta_{q}(A,B)$, so it suffices to prove \eqref{eq: N bil bis} when $p\geq2$. Fix $p\geq2$ and $A,B\in\cA(\Omega)$ such that $\Delta_{p}(A,B)>0$.

Let $\cQ=\cQ_{\delta}$ as in \eqref{eq: N Bellman}. Fix $\delta>0$ such that Theorem~\ref{t: N B gen conv} holds true. Let $\{\cP_{n}\}_{\in\N_{+}}$ be the sequence of Section~\ref{s: def Pn}. For $n\in\N_{+}$ and $\nu\in (0,1]$, define $\cR_{n,\nu}$ by means of \eqref{eq: d Rnv} and fix  $C_{1}>0$ not depending on $\nu$ such that Theorem~\ref{t: N approx} holds true.

\medskip
We now start the heat-flow method of Section~\ref{s: N heat-flow}, but for simplicity we omit the subscript $\cW$. Fix $f,g\in (L^{p}\cap L^{q})(\Omega)$. Define
$$
\cE(t)=\int_{\Omega}\cQ\left(T^{A,\oV}_{t}f,T^{B,\oV^{\prime}}_{t}g\right),\quad t>0.
$$
The estimates \eqref{eq: N 5} and the analyticity of $(T^{A,\oV}_{t})_{t>0}$ and $(T^{B,\oV^{\prime}}_{t})_{t>0}$ (see Lemma~\ref{l: N analytic sem}) imply that $\cE$ is well defined, continuous on $[0,\infty)$, differentiable on $(0,\infty)$ with a continuous derivative and
$$
-\cE^{\prime}(t)=2\Re\int_{\Omega}\left(\partial_{\zeta}\cQ\left(T^{A,\oV}_{t}f,T^{B,\oV^{\prime}}_{t}g\right)\oL^{A}T^{A,\oV}_{t}f+\partial_{\eta}\cQ\left(T^{A,\oV}_{t}f,T^{B,\oV^{\prime}}_{t}g\right)\oL^{B}T^{B,\oV^{\prime}}_{t}g\right).
$$
Integrating in the variable $t$ from $0$ to $+\infty$ both sides of the equality above, using the first estimate in \eqref{eq: N 5} and the fact that, by analyticity, $T^{A,\oV}_{t}f\in\Dom(\oL^{A}_{p})\cap\Dom(\oL^{A}_{q})$ and $T^{B,\oV^{\prime}}_{t}g\in \Dom(\oL^{B}_{p})\cap\Dom(\oL^{B}_{q})$, we deduce that for proving \eqref{eq: N bil bis} it suffices to show that
\begin{equation}
\label{eq: N by parts}
\int_{\Omega}\mod{\nabla u}\mod{\nabla v}\leqsim 2\Re\int_{\Omega}\left(\partial_{\zeta}\cQ\left(u,v\right)\oL^{A}u+\partial_{\eta}\cQ\left(u,v\right)\oL^{B}v\right),
\end{equation}
for all $u\in\Dom(\oL^{A}_{p})\cap\Dom(\oL^{A}_{q})$ and all $v\in\Dom(\oL^{B}_{p})\cap\Dom(\oL^{B}_{q})$.

Note that 
$$
\Dom(\oL^{A}_{p})\cap\Dom(\oL^{A}_{q})\subset \Dom(\oL^{A}_{2})\subset \oV
$$
and 
$$
\Dom(\oL^{B}_{p})\cap\Dom(\oL^{B}_{q})\subset \Dom(\oL^{B}_{2})\subset \oV^{\prime}.
$$
Therefore for proving \eqref{eq: N by parts} it suffices to assume that $u\in\oV$, $v\in \oV^{\prime}$ and  $u,v,\oL^{A}u, \oL^{B}v\in L^{p}(\Omega)\cap L^{q}(\Omega)$. By using Theorem~\ref{t: N approx}~(\ref{eq: R2})~and~(\ref{eq: R4}), Lemma~\ref{l: N second order}~(\ref{eq: est Q*phi i)}), the fact that $\cQ\in C^{1}(\R^{4})$ and Lebesgue's dominated convergence theorem twice, we deduce that

\begin{equation}\label{eq: N final}
\aligned
&2\Re\int_\Omega\left( (\partial_\zeta \cQ)(u,v)\oL^{A}u+ (\partial_\eta \cQ)(u,v)\oL^{B} v\right)\\
&=\lim_{\nu\rightarrow 0_{+}}\lim_{n\rightarrow+\infty}2\Re\int_\Omega \left((\partial_{\zeta}\cR_{n,\nu})(u,v)\oL^{A} u+(\partial_{\eta}\cR_{n,\nu})(u,v)\oL^{B}v\right).
\endaligned
\end{equation}
By Lemma~\ref{l: invariance} we have 
$(\partial_{\zeta}\cR_{n,\nu})(u,v)\in\oV$ and $(\partial_{\eta}\cR_{n,\nu})(u,v)\in\oV^{\prime}$. Hence we can integrate by parts the integral on the right-hand side of \eqref{eq: N final} and, by means of the chain-rule for the composition of smooth functions with vector-valued Sobolev functions, deduce that
\begin{equation}\label{eq: N final bis}
\aligned
&2\Re\int_\Omega \left((\partial_{\zeta}\cR_{n,\nu})(u,v)\oL^{A}u+(\partial_{\eta}\cR_{n,\nu})(u,v)\oL^{B}v\right)\\
&=\int_{\Omega}H^{(A,B)}_{\cR_{n,\nu}}\left[\cW_{2,1}\left(u,v\right);\cW_{2,d}\left(\nabla u,\nabla v\right)\right].
\endaligned
\end{equation}
By Theorem~\ref{t: N approx}, the function $\cR_{n,\nu}$ is $(A,B)$-convex in $\R^{4}$, so the integrand on the right-hand side of \eqref{eq: N final bis} is nonnegative for all $n\in\N_{+}$. Hence, by  Fatou's lemma, Theorem~\ref{t: N approx} (\ref{eq: R2}) and Proposition~\ref{p: N gen conv reg},
$$
\aligned
\lim_{n\rightarrow+\infty}2\Re&\int_\Omega \left((\partial_{\zeta}\cR_{n,\nu})(u,v)\oL^{A}u+(\partial_{\eta}\cR_{n,\nu})(u,v)\oL^{B}v\right)\\
&\geq \int_{\Omega}H^{(A,B)}_{\cQ\star\varphi_{\nu}}\left[\cW_{2,1}\left(u,v\right);\cW_{2,d}\left(\nabla u,\nabla v\right)\right]\\
&\geq C_{0}\int_{\Omega}|\nabla u||\nabla v|,
\endaligned
$$
for all $\nu\in (0,1)$, where $C_{0}=C_{0}(p,\Delta_{p}(A,B),\lambda(A,B),\Lambda(A,B))>0$ does not depend on $\nu$. The desired inequality \eqref{eq: N by parts} now follows from \eqref{eq: N final}.

\subsection{Remark on the special case $\Omega=\R^{d}$}\label{s: N bil RN}
In this section we simplify the proof of \cite[Theorem~1.1]{CD-DivForm} by means of elliptic regularity theory \cite{AgDoNi} and a reduction argument in the spirit of \cite[Section~7]{CD-OU}. 
\smallskip

Let  $p>1$ and $q=p/(p-1)$. By the regularisation trick explained in \cite[Lemma~A4 and Lemma~A5]{CD-DivForm}, we may assume that $A,B\in C^{\infty}(\R^{d};\C^{d\times d})$ with bounded derivatives. In this case, by elliptic regularity \cite{AgDoNi} the semigroups $(T^{A}_{t})_{t>0}$ and $(T^{B}_{t})_{t>0}$ are analytic in $L^{r}(\R^{d})$ and $\Dom(\oL^{A}_{r})=\Dom(\oL^{B}_{r})=W^{2,r}(\R^{d})$, for all $r\in (1,\infty)$; see \cite[Theorem~3.1.1 and Theorem~3.2.2]{Lunardib1}, \cite[Section~6]{KuW},   and \cite[Chapter~7]{pazy}.

Fix $f,g\in(L^{p}\cap L^{q})(\R^{d})$ and start the heat-flow method as in Section~\ref{s: N proof bil}. Since $T^{A}_{t}f,T^{B}_{t}g\in (W^{2,p}\cap W^{2,q})(\R^{d})$ for all $t>0$, for proving \eqref{eq: N bil RN} it suffices to show that
\begin{equation}\label{eq: N bil bis bis}
\int_{\R^{d}}\mod{\nabla u}\mod{\nabla v}\leqsim 2\Re\int_{\R^{d}}\left((\partial_{\zeta}\cQ)\left(u,v\right)\oL^{A}_{p}u+(\partial_{\eta}\cQ)\left(u,v\right)\oL^{B}_{q}v\right),
\end{equation}
for all  $u,v\in(W^{2,p}\cap W^{2,q})(\R^{d})$.

Take $u_{n},v_{n}\in C^{\infty}_{c}(\R^{d})$ such that $u_{n}\rightarrow u$ in $W^{2,p}(\R^{d})=\Dom(\oL^{A}_{p})$ and $v_{n}\rightarrow v$ in $W^{2,q}(\R^{d})=\Dom(\oL^{B}_{q})$. By \eqref{eq: N 5} the sequence $(\partial_{\zeta}\cQ)(u_{n},v_{n})$ is bounded in $L^{q}(\R^{d})$ and the sequence $(\partial_{\eta}\cQ)(u_{n},v_{n})$ is bounded in $L^{p}(\R^{n})$. By passing to  subsequences, we may assume that $u_{n}\rightarrow u$ and $v_{n}\rightarrow v$ almost everywhere in $\R^{d}$, so that $(\partial_{\zeta}\cQ)(u_{n},v_{n})\rightarrow(\partial_{\zeta}\cQ)(u,v)$ weakly in $L^{q}(\R^{d})$ and $(\partial_{\eta}\cQ)(u_{n},v_{n})\rightarrow(\partial_{\eta}\cQ)(u,v)$ weakly in $L^{p}(\R^{d})$. It follows that it suffices to prove \eqref{eq: N bil bis bis} for all $u,v\in C^{\infty}_{c}(\R^{d})$ (alternatively, one can arrive at the very same conclusion by using the analogue of \cite[Lemma~29]{CD-OU} which is obtained by replacing $|G|+1$ with $|G|+\gamma$, $\gamma(x)=e^{-|x|^{2}}$). 

Fix $u,v\in C^{\infty}_{c}(\R^{d})$. Recall that we are assuming that $A,B$ are smooth. By Lemma~\ref{l: N second order}~(\ref{eq: est Q*phi i)}) and Lebesgue convergence theorem,
$$
\aligned
2\Re\int_{\R^{d}}&\left(\partial_\zeta \cQ(u,v)\oL^{A}u+ \partial_\eta \cQ(u,v)\oL^{B} v\right)\\
&=\lim_{\nu\rightarrow 0_{+}}2\Re\int_{\R^{d}} \left(-\partial_{\zeta}(\cQ\star\varphi_{\nu})(u,v)\div(A\nabla u)-\partial_{\eta}(\cQ\star\varphi_{\nu})(u,v)\div(B\nabla v)\right),
\endaligned
$$
A standard integration by parts and Proposition~\ref{p: N gen conv reg} now give
$$
\aligned
2\Re\int_{\R^{d}}\Big(\partial_\zeta \cQ(u,v)  & \oL^{A}u+ \partial_\eta \cQ(u,v)\oL^{B} v\Big)\\
&\geq \liminf_{\nu\rightarrow 0_{+}} \int_{\R^{d}}H^{(A,B)}_{\cQ\star\varphi_{\nu}}\left[\cW_{2,1}\left(u,v\right);\cW_{2,d}\left(\nabla u,\nabla v\right)\right]\\
&\geqsim \int_{\R^{d}}\mod{\nabla u}\mod{\nabla v}
\endaligned
$$
as required for finishing the proof.
\section{Maximal regularity and functional calculus: proof of Theorem~\ref{t: N principal}}\label{s: max funct}
The interested reader should consult the monographs \cite{KuW,DHMP} and \cite{Haase} for a detailed discussion on the maximal regularity problem for generators of analytic semigroups on Banach spaces; below we shortly describe the problem and recall the principal results we need for proving Theorem~\ref{t: N principal}.

\subsection{Maximal regularity}Let $\sX$ be a complex Banach space and $-\oA$ the generator of a strongly continuous semigroup on $\sX$. Let $\tau>0$ and $r\in (1,+\infty)$. 

We say that  $\oA$ has {\it maximal $L^{r}$-regularity in $(0,\tau)$} if for every $v\in L^{r}(0,\tau;{\mathfrak X})$ the unique {\it mild solution}
\begin{equation}\label{eq: N mild sol}
u(t):=\int^{t}_{0}e^{-(t-s)\oA}v(s)\wrt s,\quad (0<t<\tau)
\end{equation}
to the Cauchy problem $u^{\prime}(t)+\oA u(t)=v(t)$, $u(0)=0$ belongs to $W^{1,r}(0,\tau;\sX)\cap L^{r}(0,\tau;\Dom(\oA))$. This property does not depend on $\tau>0$ and $r>1$ \cite[Theorem~2.5]{Dore}. We say that $\oA$ has {\it (parabolic) maximal regularity} if for some, equivalently all, $r>1$ and $\tau>0$ the operator $\oA$ has maximal $L^{r}$-regularity in $(0,\tau)$. It follows from the very definition that $\oA$ has maximal regularity if and only if $\sigma I+\oA$ has maximal regularity for all $\sigma\geq 0$; see \cite[p. 29]{Dore}. Also, if $\oA$ has maximal regularity, then there exists $\sigma\geq 0$ such that $\exp(-t(\sigma+\oA))_{t>0}$ is bounded and analytic in $\sX$; see \cite[Theorem~2.2]{Dore}. 
\medskip

Suppose that $-\oA$ is the generator of a bounded analytic semigroup on a reflexive Banach space $\sX$, that is to say, assume that $\oA$ is {\it sectorial} with {\it sectoriality angle} $\omega(\oA)<\pi/2$ \cite{CDMY}. Denote respectively by ${\rm N}(\oA)$ and ${\rm R}(\oA)$ the nullspace and the range of $\oA$. By \cite[Theorem~3.8]{CDMY}, we have 
\begin{equation}\label{eq: CDMY}
\sX={\rm N}(\oA)\oplus\overline{{\rm R}}(\oA),
\end{equation}
where the sum is direct. 
\smallskip

As a consequence, we can always factor off the nullspace of $\oA$ and study maximal regularity for $v\in L^{r}(0,\tau;\sX)$ such that $v(s)\in \ovR(\oA)$, for a.e. $s\in (0,\tau)$. 
\subsection{Functional calculus}
Consider a reflexive complex Banach space $\sX$ and the generator $-\oA$ of a bounded analytic semigroup on $\sX$. By \eqref{eq: CDMY}, the restriction $\oA_{\vert\vert}$ of $\oA$ to $\overline{{\rm R}}(\oA)$ is a densely defined one-to-one sectorial operator with dense range on the Banach space $\overline{{\rm R}}(\oA)$, with sectoriality angle
 $\omega(\oA)<\pi/2$  and the functional calculus introduced in \cite{CDMY} is applicable to it. In particular, for every $\theta\in (\omega(\oA),\pi)$ and every bounded and holomorphic function $m$ in the cone $\bS_{\theta}=\{z\in\C\setminus\{0\}: |\arg(z)|<\theta|\}$ (in short, for every $m\in H^{\infty}(\bS_{\theta})$) we may define the closed densely defined, but possibly unbounded, linear operator $m\left(\oA_{\vert\vert}\right)$. We refer the interested reader to \cite{Mc, CDMY, Haase} for an exhaustive treatment of this subject.
 
Let $\theta\in (\omega(\oA),\pi)$. We say that $\oA$ admits a {\it bounded $H^{\infty}(\bS_{\theta})$-calculus} if $m(\oA_{\vert\vert})$ is bounded on $\ovR(\oA)$ whenever $m\in H^{\infty}(\bS_{\theta})$. We say that $\oA$ has a {\it bounded $H^{\infty}$-calculus} if it has a bounded $H^{\infty}(\bS_{\theta})$-calculus for some $\theta>\omega(\oA)$. The functional calculus angle $\omega_{H^{\infty}}(\oA)$ is, by definition, the infimum over all angles $\theta>0$ such that $\oA$ has a bounded $H^{\infty}(\bS_{\theta})$-calculus (with the convention that $\omega_{H^{\infty}}(\oA)=+\infty$ if $\oA$ does not have a bounded $H^{\infty}$-calculus). 

It is an interesting and widely studied problem whether a sectorial operator on a Banach space has a bounded $H^{\infty}$-calculus, and it is of interest to explicitly determine or estimate the functional calculus angle of the operator; see \cite{cowling,Mc,McY,CDMY,KW,KuW,Kalton}, \cite{CD-mult,CD-OU} and the references contained therein. 

In the special case when $\sX=\sH$ is a Hilbert space,
by a universal result of McIntosh~\cite{Mc} we always have $\omega_{H^{\infty}}(\oA)=\omega(\oA)$. The norm of $m(\oA)$, $m\in H^{\infty}(\bS_{\theta})$, $\theta>\omega(\oA)$, may depend on $\theta$, the space $\sH$ and the operator $\oA$. However, by a universal result of Crouzeix and Delyon \cite{CrDe}, it is always bounded above by $(2+2/\sqrt{3})\norm{m}{\infty}$ whenever $m\in H^{\infty}(\bS_{\theta})$ and $\theta>\sup\left\{\arg(\sk{\oA u}{u}): u\in \Dom(\oA)\right\}$ (the numerical range angle of $\oA$).
\medskip

One reason for studying the boundedness of $H^{\infty}$-calculus for sectorial operators on Banach spaces is its close tie with the maximal regularity problem. 
\medskip

Recall Lemma~\ref{l: N analytic sem}. In the context of the present paper, by either using the Dore-Venni theorem \cite{DoreVenni} in the refined form of Pr\"uss and Sohr \cite{PrussSohr} (see also \cite{GigaSohr}), or the characterisation of maximal regularity by Weis \cite{Weis1} together with the theory developed by Kalton and Weis in \cite{KW}, we obtain the following result.
\begin{proposition}\label{p: N D-V bis}
Let $W^{1,2}_{0}(\Omega)\subseteq \oV\subseteq W^{1,2}(\Omega)$ as in Section~\ref{s: boundary}. Suppose that $A\in\cA(\Omega)$, $p>1$ and $\Delta_{p}(A)>0$. Let $\oL^{A}_{p}$ denote the negative generator of $(T^{A,\oV}_{t})_{t>0}$ in $L^{p}(\Omega)$.
If $\omega_{H^{\infty}}(\oL^{A}_{p})<\pi/2$, then $\oL^{A}_{p}$ has parabolic maximal regularity.
 \end{proposition} 
\subsection{Proof of Theorem~\ref{t: N principal}}
Let $p>1$ and $A\in\cA(\Omega)$. In light of Proposition~\ref{p: N D-V bis} it suffices to show that 
$$
\Delta_{p}(A)>0\implies \omega_{H^{\infty}}(\oL^{A}_{p})<\pi/2.
$$ 
By Lemma~\ref{l: N prop of Deltap}~(\ref{L4 (i)}), \ref{L4 (ii)}) and a standard duality argument we may assume that $p\geq 2$.

By Lemma~\ref{l: N prop of Deltap}~(\ref{L4 (ii)}), (\ref{L4 (iv)}) there exists $\theta\in (0,\pi/2)$ such that $\Delta_{p}\left(e^{\pm i\theta}A,e^{\mp i\theta}A^{*}\right)>0$. Then, by Remark~\ref{r: N analytic sem}, for every $r\in [q,p]$ both $(T^{A,\oV}_{t})_{t>0}$ and $(T^{A^{*}\oV}_{t})_{t>0}$ are analytic (and contractive) in $L^{r}(\Omega)$ in the cone $\bS_{\theta}$.
\smallskip

Moreover, by Theorem~\ref{t: N bil} there exists $C>0$ such that   
\begin{equation}\label{eq: bilineq*}
\int^{\infty}_{0}\!\int_{\Omega}\mod{\nabla T^{A,\oV}_{te^{\pm i\theta}}f(x)}\mod{\nabla T^{A^{*},\oV}_{te^{\mp i\theta}}g(x)}\wrt x\wrt t\leq C \norm{f}{p}\norm{g}{q},
\end{equation}
for all $f,g\in (L^{p}\cap L^{q})(\Omega)$.

It follows from \eqref{eq: bilineq*} and the inequality
$$
\mod{\int_{\Omega}\oL^{A}_{2}T^{A,\oV}_{te^{\pm i\theta}}f\,\overline{T^{A^{*},\oV}_{te^{\mp i\theta}}g }\wrt x}\leq \Lambda(A)\int_{\Omega}\mod{\nabla T^{A,\oV}_{te^{\pm i\theta}}f(x)}\mod{\nabla T^{A^{*},\oV}_{te^{\mp i\theta}}g(x)}\wrt x
$$
that
$$
\int^{\infty}_{0}\!\mod{\int_{\Omega}\oL^{A}_{p}T^{A,\oV}_{2te^{\pm i\theta}}f\, \overline{g }\wrt x}\wrt t\leq C\Lambda(A)\norm{f}{p}\norm{g}{q},
$$
for all $f,g\in (L^{p}\cap L^{q})(\Omega)$. Analyticity of $(T^{A,\oV}_{t})_{t>0}$ in $L^{p}(\Omega)$, Fatou's lemma and a density argument show that
\begin{equation}\label{eq: N CDMY}
\int^{\infty}_{0}\!\mod{\int_{\Omega}\oL^{A}_{p}T^{A,\oV}_{te^{\pm i\theta}}f\, \overline{g }\wrt x}\wrt t \leq 2 \Lambda(A)C\norm{f}{p}\norm{g}{q},
\end{equation}
for all $f\in L^{p}(\Omega)$ and all $g\in L^{q}(\Omega)$. 

We now apply \cite[Theorem~4.6 and example~4.8]{CDMY} to the dual subpair $\sk{\ovR(\oL^{A}_{p})}{\ovR(\oL^{A^{*}}_{q})}$ and the dual operators $(\oL^{A}_{p})_{\vert\vert}$, $(\oL^{A^{*}}_{q})_{\vert\vert}$, and deduce from \eqref{eq: N CDMY} that $\omega_{H^{\infty}}(\oL^{A}_{p})\leq \pi/2-\theta$.

\numberwithin{theorem}{section}
\numberwithin{equation}{section}

\appendix  

\section{Comparison with known results}\label{s: comparison} 
Under stronger assumptions on $A$ and/or $\Omega$ than those of Theorem~\ref{t: N principal} it is known that $(T^{A,\oV}_{t})_{t>0}$, where $\oV$ is one of the subspaces of Section~\ref{s: boundary}, extrapolates to a bounded strongly continuous semigroup on $L^{p}(\Omega)$ in a range of $p$'s larger than the range given by $p$-ellipticity, and the negative generator $\oL^{A}_{p}$ has parabolic maximal regularity in this larger range of exponents. 

\subsection{Semigroup extrapolation}
Let $\oV$ denote one of the subspaces of Section~\ref{s: boundary}.
\medskip

(i) For every $\Omega\subseteq\R^{d}$ and every {\bf real-valued} $A\in\cA(\Omega)$ the semigroup $(T^{A,\oV}_{t})_{t>0}$ is sub-Markovian (see \cite{Ouh1992,Ouh1996} and  \cite[Corollary~4.3 and Corollary~4.10]{O}), so $\omega_{H^{\infty}}(\oL^{A}_{p})<\pi/2$ for all $p>1$; see \cite{cowling}, \cite{KS}, \cite{CD-mult} for symmetric real-valued $A$ and \cite[Corollary~5.2]{KW} for nonsymmetric real-valued $A$. It follows from Dore-Venni theorem \cite{DoreVenni, PrussSohr} that, in this case, $\oL^{A}_{p}$ has parabolic maximal regularity for all $p>1$. 
\smallskip

(ii) As for {\bf complex-valued} $A\in\cA(\Omega)$, define the upper and lower Sobolev exponent by the rule $2^{*}=2d/(d-2)$ if $d\geq 3$ and $2^{*}=+\infty$ if $d=1,2$;  $2_{*}=(2^{*})^{\prime}$. For $A\in\cA(\Omega)$, $\delta\geq 0$ and $\theta\in [0,\pi/2)$ denote by $J(A,\oV,\delta,\theta)$ the maximal open interval in $(1,+\infty)$ such that $(e^{-\delta z}T^{A,\oV}_{z})_{z\in{\bf S}_{\theta}}$ is uniformly bounded in $L^{p}(\Omega)$, for all $p\in J(A,\oV,\delta,\theta)$. Denote by $\omega(\oL^{A}_{2})$ the sectoriality angle of $\oL^{A}_{2}$.

\begin{itemize}
\item[(a)] When $\Omega=\R^{d}$, $d=1,2$ and $A\in\cA(\R^{d})$ we have $J(A,W^{1,2}(\R^{d}),0,\theta)=(1,+\infty)$, for all $\theta>\pi/2-\omega(\oL^{A}_{2})$ \cite{AMT}. When $\Omega=\R^{d}$, $d\geq 3$ and $A\in\cA(\R^{d})$ Auscher proved \cite{Auscher1} that there exists $\epsilon>0$ depending only on dimension and the ellipticity constants of $A$ such that $(2_{*}-\epsilon,2^{*}+\epsilon)\subset J(A,W^{1,2}(\R^{d}),0,\theta)$, for all $\theta>\pi/2-\omega(\oL^{A}_{2})$.

\item[(b)] The results in (a) are sharp \cite{HMMcl}: for all $d\geq 3$ and all $p\in [1,\infty]\setminus[2_{*},2^{*}]$ there exists $A\in\cA(\R^{d})$ such that $(T^{A}_{t})_{t>0}$ is not bounded on $L^{p}(\R^{d})$.
\item[(c)] In the case when $\oV$ has the embedding property \eqref{eq: N S embedding} with $q=2^{*}$ Egert implicitly\footnote{In \cite{Egert} the author also considered  systems. The results are stated under geometric assumptions on $\Omega$ which are stronger than \eqref{eq: N S embedding} with $q=2^{*}$ . These stronger assumptions are used, for example, for proving results on Riesz transforms. However, for the specific result stated here \eqref{eq: N S embedding} with $q=2^{*}$ suffices; see \cite{Egert2018}.} proved in \cite[Theorem~1.6]{Egert} that $(2_{*},2^{*})\subset J(A,\oV,\delta,\theta)$, for every $\delta>0$ and $\theta>\pi/2-\omega(\oL^{A}_{2})$. Also, Egert extended the result in (a) ($d\geq 3$) by proving  that if $\Omega$ is bounded and connected, the boundary is Lipschitz regular around the Neumann part $\overline{\partial\Omega\setminus D}$ and $D$ satisfies further geometric assumptions (see \cite{Egert, H-DJKR, Elst2019})
one always has that $(2_{*}-\epsilon,2^{*}+\epsilon)\subset J(A,\oV,\delta,\theta)$, for all $\delta>0$ and $\theta>\pi/2-\omega(\oL^{A}_{2})$ and some $\epsilon>0$ depending only on dimension, the ellipticity constants of $A$ and the geometry of $\Omega$. Note that under the above mentioned geometric assumptions there exists a Sobolev extension operator $E: \oV\rightarrow W^{1,2}(\R^{d})$ and so \eqref{eq: N S embedding} holds true with $q=2^{*}$.
\end{itemize}

A result similar to (c), but without any further geometric assumption on $D$, has been previously obtained by Tolksdorf \cite{Tolks}, who also proved maximal regularity in the range $(2_{*}-\epsilon,2^{*}+\epsilon)$. 
\smallskip

Recently, Egert improved the result in (c) by combining our notion of $p$-ellipticity and its properties with the technology developed in \cite{Auscher1,Egert}. A similar result has been proved by ter Elst, Haller-Dintelmann, Rehberg and Tolksdorf \cite[Theorem~3.1]{Elst2019} by means of a different technique, but still using $p$-ellipticity.
\smallskip
 
For $p\geq 2$ set
$$
p^{\circ}:=\frac{p}{2}2^{*}=\frac{pd}{d-2}.
$$
\begin{proposition}[{\cite[Theorem~1]{Egert2018}}]\label{p: Egert 2019}
Let $d\geq 3$. Let $\Omega\subset\R^{d}$ be open. Let $\oV$ denote one of the subspaces of Section~\ref{s: boundary}. Assume the Sobolev embedding \eqref{eq: N S embedding} with $q=2^{*}$. Let $p^{}_{0}>2$ be such that $\Delta_{p_{0}}(A)>0$. Then
$$
\left[(p^{\circ}_{0})^{\prime},p^{\circ}_{0}\right]\subset J(A,\oV,\delta,\theta),
$$
for all $\delta>0$ and $\theta>\pi/2-\omega(\oL^{A}_{2})$.
\end{proposition}
\subsection{Functional calculus and maximal regularity}
Next result follows by combining Proposition~\ref{p: Egert 2019} with a result of Blunck and Kunstmann \cite{BK} that was simplified by Auscher in \cite[Theorem~1.1]{Auscher1} and extended to domains $\Omega\subset\R^{d}$ by Egert in \cite[Proposition~5.2]{Egert}. 
\begin{corollary}[\cite{Egert,Egert2018}]\label{c: Egert 2018}
Under the assumptions of Proposition~\ref{p: Egert 2019} we have  
$$
\omega_{H^{\infty}}(\oL^{A}_{p}+\delta I)=\omega(\oL^{A}_{2})<\pi/2,
$$ 
for every $p\in ((p^{\circ}_{0})^{\prime},p^{\circ}_{0})$ and $\delta>0$.

As a consequence \cite{DoreVenni, PrussSohr}, $\oL^{A}_{p}$ has parabolic maximal regularity, for every $p\in [(p^{\circ}_{0})^{\prime},p^{\circ}_{0}]$. 
\end{corollary}

\subsection{Absence of Sobolev embeddings}\label{s: N nonSobolev} 
It is well-known that \eqref{eq: N S embedding} for $q>2$ requires geometric assumptions on $\Omega$ and does not hold in general, see \cite[Theorem~4.46, Theorem~4.48 and Example~4.55]{AdFour} and \cite[Proposition~3 and Example~6]{BuDa}. For simplicity, we only discuss the case when $\oV=W^{1,2}(\Omega)$.
\smallskip

When $\Omega\subset\R^{d}$ has finite measure, by the Rellich-Kondrachov theorem \cite[Theorem~5 and Corollary~1]{PioPekka}, the Sobolev embedding \eqref{eq: N S embedding} for some $q>2$ implies the compactness of the resolvent operator $(I+\oL^{A}_{2})^{-1}$ for all $A\in\cA(\Omega)$ (see also \cite[Theorem~7]{BuDa}). So, in this case, the spectrum of $\oL^{I}_{2}$ is discrete.
\begin{enumerate}[(a)]
\item A classical example of Courant and Hilbert \cite[p. 531]{CourantHilbert} shows that there exists a ``rooms and passages'' connected bounded domain $\Omega\subset\R^{2}$ for which $0\in\sigma_{\rm ess}(\oL^{I}_{2})$.  Actually, for $d\geq 2$, given any closed subset $S$ of $[0,+\infty)$ there exists an open connected subset $\Omega$ of the unit ball in $\R^{d}$ such that $\sigma_{\rm ess}(\oL^{I}_{2})=S$, see \cite{HeSeSi}.
\item Let $d\geq 2$. By using a criterion of Evans and Harris \cite{EvHa} (see also \cite[p. 106]{DaSi1}), one can easily construct unbounded ``horn-shaped'' $\Omega\subset\R^{d}$ of finite measure for which the Neumann Laplacian $\oL^{I}_{2}$ does not have compact resolvent operators. A simple example of this phenomenon is given by the regions
\begin{equation}\label{eq: Omegaalpha}
\Omega_{\alpha}=\{(x,x^{\prime})\in \R\times\R^{d-1}: x>0,\ |x^{\prime}|<ce^{-\alpha x}\},\quad \alpha>0,\ c>0.
\end{equation}
\end{enumerate}

\subsection{Sharpness of Theorem~\ref{t: N principal}}\label{s: sharp} For open sets like those described above in (b), the conclusions of Proposition~\ref{p: Egert 2019} and Corollary~\ref{c: Egert 2018} are false, because the analyticity angle of the semigroup and the functional calculus angle of the generator may depend on $p$, even for smooth $A$ and pure Neumann boundary conditions.

Kunstmann \cite{kunstmann2} further developed a result of Davies and Simon \cite{DaSi1} for $p=2$ and proved that for $d=2$ and $p>1$ the $L^{p}$ spectrum of the Neumann Laplacian in the region $\Omega_{\alpha}$ satisfies the inclusions
\begin{equation}\label{eq: parabola}
\{0\}\cup {\bf P}_{p,\alpha}\subseteq\sigma(\oL^{I}_{p})\subseteq \left[0,\alpha^{2}/4\right)\cup {\bf P}_{p,\alpha},
\end{equation}
where 
$$
{\bf P}_{p}=\left\{x+iy\in\C: x=\frac{p^{2}}{(p-2)^{2}}y^{2}+\frac{1}{p}-\frac{1}{p^{2}}\right\},\quad {\bf P}_{p,\alpha}=\alpha^{2}{\bf P}_{p}.
$$
Moreover, $\sigma_{\rm ess}(\oL^{I}_{p})={\bf P}_{p,\alpha}$.

For $p\in (1,+\infty)$ and $\alpha>0$ the parabolic region ${\bf P}_{p,\alpha}$ is tangent to the critical sector ${\bf S}_{\phi^{*}_{p}}$, where $\phi^{*}_{p}=\arcsin|2/p-1|$. Recall that $\phi_{p}=\pi/2-\phi^{*}_{p}$ and $\phi^{*}_{p}$ are, respectively, the optimal analyticity angle in $L^{p}$ \cite{lp,KR} and the optimal functional calculus angle in $L^{p}$ \cite{CD-mult}, for all generators of symmetric contraction semigroups.
\medskip

By attaching to, say, a ball in $\R^{2}$ countably many disjoint horns $\widetilde{\Omega}_{\alpha_{n}}$, each one congruent to some $\Omega_{\alpha_{n}}$, $\alpha_{n}>0$, Kunstmann \cite{kunstmann2} was able to construct a domain $\Omega_{\max}$ of finite measure such that the associated Neumann Laplacian has maximal $L^{p}$ spectrum: 
\begin{equation}\label{eq: max spec}
\sigma(\oL^{I}_{p})=\overline{\bS}_{\phi^{*}_{p}},\quad 1<p<+\infty.
\end{equation}
Consider the region $\Omega_{{\rm max}}$ described above. Recall from \cite[Lemma~5.22]{CD-DivForm} that for $\phi\in(-\pi/2,\pi/2)$ and $p>1$ we have
\begin{equation}\label{eq: angle}
\Delta_p(e^{i\phi}I)=\cos\phi-\mod{1-2/p}=\cos\phi-\cos\phi_{p}.
\end{equation}   
Fix $\phi\in (0,\pi/2)$. Since $(T^{I}_{t})_{t>0}$ is symmetric and sub-Markovian, we obtain from \cite{CD-mult}, \eqref{eq: angle} and \eqref{eq: max spec} that 
$$
\left\{p\in (1,+\infty): \omega_{H^{\infty}}(\oL^{e^{i\phi}I}_{p})<\pi/2\right\}=\left\{p\in (1,+\infty):\Delta_{p}(e^{i\phi}I)>0\right\}.
$$

Let $p>1$ be such that $\Delta_{p}(e^{i\phi}I)\leq 0$, that is, such that $\phi\geq\phi_{p}$. Then $\oL^{e^{i\phi} I}_{p}$ does not have parabolic maximal regularity in $L^{p}(\Omega_{{\rm max}})$, since otherwise, by \cite[Theorem~2.2]{Dore}, there would exist $\epsilon,\delta>0$ such that 
$
\sigma(\oL^{I}_{p})\subseteq \overline{\bS}_{\pi/2-\phi-\epsilon}-\delta
$, contradicting \eqref{eq: max spec}. 

More generally, by combining \cite{CD-mult} with Proposition~\ref{p: N D-V bis} and \cite[Example~2.4 and Remark~3.2]{Kunstmann3}, we deduce that for all $d\geq2$ there exist $\Omega\subset\R^{d}$ and $A\in\cA(\Omega)$ such that the Neumann operator $\oL^{A}_{p}$ has parabolic maximal regularity if and only if $\Delta_{p}(A)>0$.

This shows that Theorem~\ref{t: N principal} is sharp.
\subsection{Extrapolation in $L^{\infty}$ for smooth coefficients: counterexamples}\label{s: extrapolation ex} Consider the open set $\Omega_{\rm max}$ of \cite{kunstmann2} described above. The equality \eqref{eq: max spec} implies that if $\phi>0$ then $(T^{e^{i\phi}I}_{t})_{t>0}$=$(T^{I}_{e^{i\phi}t})_{t>0}$ is {\sl not} exponentially bounded in $L^{\infty}(\Omega)$. Indeed, assuming the contrary, by interpolation with $L^{2}$ and the relation $\overline{T^{I}_{z}f}=T^{I}_{\bar z}\bar f$, $\Re z>0$, (because $(T^{I}_{t})_{t>0}$ is positive and analytic in $L^{2}$ in $\{\Re z>0\}$), there would exist $r_{0}>0$ and $\phi_{0}\in (0,\pi/2)$ such that 
$$
\sup_{t>0}\norm{e^{-r_{0}t}T^{I}_{e^{\pm i\phi_{0}}t}}{p}<+\infty,\quad \forall p>2.
$$
This implies that $\sigma(\oL^{I}_{p})\subseteq \overline{{\bf S}}_{\pi/2-\phi_{0}}-r_{0}$ for all $p>2$, contradicting \eqref{eq: max spec} when $p$ is such that $\phi^{*}_{p}>\pi/2-\phi_{0}$. 

\begin{example}
We would expect that if $\phi\neq 0$ then one has $T^{I}_{e^{i\phi}t}(L^{\infty}(\Omega_{\rm max}))\not\subset L^{\infty}(\Omega_{\rm max})$ for some $t>0$, but we could not extract this result from the existing literature. Therefore, we now construct a (disconnected) open set $\Omega\subset\R^2$ such that there exists $\phi>0$ for which $T^{e^{i\phi}I}_{1}(L^{\infty}(\Omega))\not\subset L^{\infty}(\Omega)$.

We first consider the Neumann Laplacian $\Delta^{\Omega_{\alpha}}:=\oL^{I}$ in the region $\Omega_{\alpha}$ given by \eqref{eq: Omegaalpha} with $d=2$, $\alpha>0$ and $c=\alpha^{-1}$. Note that $\Omega_{\alpha}=\alpha^{-1}\cdot\Omega_{1}$. Set $T^{\Omega_{\alpha}}_{t}=\exp(-t\Delta^{\Omega_{\alpha}})$.

By arguing much as in the case of $\Omega_{\rm max}$ discussed above and using \eqref{eq: parabola}, we see that 
$$
\sup_{t>0}\norm{T^{\Omega_{1}}_{e^{i\phi}t}}{\infty}=+\infty
$$
for every $\phi\in (0,\pi/2)$. Fix $\phi\in (0,\pi/2)$. By the uniform boundeness principle, there exists a nonzero $f_{1}\in L^{\infty}(\Omega_{1})$ and a sequence $t_{n}>0$ such that
$$
\norm{T^{\Omega_{1}}_{t^{2}_{n}e^{i\phi}}f_{1}}{L^{\infty}(\Omega_{1})}\geq n\, \norm{f_{1}}{L^{\infty}(\Omega_{1})},\quad \forall n\in\N_{+}.
$$
We now use a rescaling argument. For $\alpha>0$, consider the operator
$$
(J_{\alpha}f)(x,y):=\alpha f(\alpha x,\alpha y).
$$
It is not hard to see that
\begin{itemize}
\item[(\rm i)] $J_{\alpha}:L^{2}(\Omega_{1})\rightarrow L^{2}(\Omega_{\alpha})$ is a surjective isometry with $J^{-1}_{\alpha}=J_{1/\alpha}$;
\item[(\rm ii)] $J_{\alpha}\Dom(\Delta^{\Omega_{1}}_{2})=\Dom(\Delta^{\Omega_{\alpha}}_{2})$;
\item[(\rm iii)] $J^{-1}_{\alpha}\Delta^{\Omega_{\alpha}}_{2}J_{\alpha}=\alpha^{2}\Delta^{\Omega_{1}}_{2}$.
\end{itemize}
It follows that 
$$
J_{1/\alpha}T^{\Omega_{\alpha}}_{z}J_{\alpha}=T^{\Omega_{1}}_{\alpha^{2}z},\quad \forall z\in\C_{+},\quad \forall \alpha>0.
$$
Hence, for $\alpha=t_{n}$ and $z=e^{i\phi}$,
$$
\left(T^{\Omega_{t_{n}}}_{e^{i\phi}}\left(f_{1}(t_{n}\cdot\right)\right)(\tfrac{\cdot}{t_{n}})=T^{\Omega_{1}}_{t^{2}_{n}e^{i\phi}}\left(f_{1}\right)(\cdot).
$$
Set $f_{n}=f_{1}(t_{n}\cdot)$. Then $\norm{f_{n}}{L^{\infty}(\Omega_{t_{n}})}=\norm{f_{1}}{L^{\infty}(\Omega_{1})}$ and
\begin{equation}\label{eq: h12}
\begin{aligned}
n\norm{f_{1}}{L^{\infty}(\Omega_{1})}\leq \|T^{\Omega_{1}}_{t^{2}_{n}e^{i\phi}}f_{1}\|_{L^{\infty}(\Omega_{1})}=\|T^{\Omega_{t_{n}}}_{e^{i\phi}}f_{n}\|_{L^{\infty}(\Omega_{t_{n}})}.
\end{aligned}
\end{equation}
For each $n\in\N_{+}$ select a rigid motion of plane $R_{n}$ such that the congruent copies 
$$
\widetilde{\Omega}_{t_{n}}:=R_{n}\left(\Omega_{t_{n}}\right),\quad n\in\N_{+}
$$ 
are pairwise disjoint. Define 
$$
\Omega:=\bigcup_{n\in\N_{+}}\widetilde{\Omega}_{t_{n}},\quad \widetilde{f}_{n}:=f_{n}(R^{-1}_{n}\cdot),\quad f=\sum_{n\in\N_{+}}\widetilde{f}_{n}\ca_{\widetilde{\Omega}_{t_{n}}}.
$$ 
Then, $\|f\|_{L^{\infty}(\Omega)}=\norm{f_{1}}{L^{\infty}(\Omega_{1})}$
and
$$
\left(T^{\Omega}_{e^{i\phi}}f\right)\ca_{\widetilde{\Omega}_{t_{n}}} =T^{\widetilde{\Omega}_{t_{n}}}_{e^{i\phi}}\widetilde{f}_{n}=\left(T^{\Omega_{t_{n}}}_{e^{i\phi}}f_{n}\right)(R^{-1}_{n}\cdot).
$$
It follows from \eqref{eq: h12} that $T^{\Omega}_{e^{i\phi}}f\notin L^{\infty}(\Omega)$.
\end{example}
\section{Rigidity of generalized convexity}
\begin{proposition}\label{p: limit of g-conv}
Let $A\in \C^{d\times d}$ be an elliptic matrix, $t_{0}>0$ and let $\gamma\in C^{1}([0,+\infty);\R)\cap C^{2}((0,+\infty)\setminus\{t_{0}\};\R)$. Set $\Gamma(\zeta)=\gamma(|\zeta|)$, $\zeta\in\R^{2}$. Suppose that $\Gamma$ is $A$-convex in $\R^{2}\setminus\{|\zeta|=t_{0}\}$. Then,
\begin{itemize}
\item[(i)] The profile function $\gamma$ is nondecreasing and convex, and $\Gamma$ is convex.
\item[(ii)] If $\Im A\neq 0$, then either $\gamma={\rm const}$ or $\gamma^{\prime}(t)>0$ for all $t>0$.
\end{itemize}
\end{proposition}
\begin{proof}
A rapid calculation shows that
\begin{equation}\label{eq: App1}
(D^{2}\Gamma)(\zeta)=\gamma^{\prime\prime}(|\zeta|) \frac{\zeta}{|\zeta|}\otimes\frac{\zeta}{|\zeta|}+\frac{\gamma^{\prime}(|\zeta|)}{|\zeta|}\left[I_{\R^{2}}-\frac{\zeta}{|\zeta|}\otimes\frac{\zeta}{|\zeta|}\right],
\end{equation}
for all $\zeta\in\R^{2}\setminus(\{|\zeta|=t_{0}\}\cup\{0\})$.
\smallskip

\underline {We first prove that $\gamma^{\prime}(t)\geq 0$ for all $t>0$.} By continuity it suffices to consider $t\neq t_{0}$. For $Z=(Z^{1},Z^{2})\in \R^{d}\times\R^{d}$ write
$$
Z^{j}=(z^{j}_{1},\dots,z^{j}_{N}),\quad j=1,2.
$$
Fix $\zeta\in\R^{2}\setminus(\{|\zeta|=t_{0}\}\cup\{0\})$. Take $X^{1}=(x^{1}_{1},0,\dots,0)$, $X^{2}=(x^{2}_{1},0,\dots,0)$. Define $X=(X^{1},X^{2})$,$Y=(Y^{1},Y^{2})=\cM(A)X$, $x=(x^{1}_{1},x^{2}_{1})$ and $y=(y^{1}_{1},y^{2}_{1})$. Then, $\sk{\cM(A)X}{X}_{\R^{2d}}=\sk{y}{x}_{\R^{2}}$ and
$$
\aligned
H^{A}_{\Gamma}[\zeta;X]&=\sk{(D^{2}\Gamma)(\zeta)x}{y}_{\R^{2}}\\
&=\frac{\gamma^{\prime}(|\zeta|)}{|\zeta|}\sk{X}{\cM(A)X}_{\R^{2d}}+\left(\gamma^{\prime\prime}(|\zeta|)-\frac{\gamma^{\prime}(\zeta)}{|\zeta|}\right)|\zeta|^{-2}\sk{\zeta}{x}_{\R^{2}}\sk{\zeta}{y}_{\R^{2}}.
\endaligned
$$
Now take $X\neq 0$ of the form above and such that $x$ is orthogonal to $\zeta$. Then, by the ellipticity of $A$, $
\sk{X}{\cM(A)X}_{\R^{2d}}\geq\lambda |X|^{2}>0
$ and by assumption of $A$-convexity of $\Gamma$,
$$
0\leq H^{A}_{\Gamma}[\zeta;X]=\frac{\gamma^{\prime}(|\zeta|)}{|\zeta|}\sk{X}{\cM(A)X}_{\R^{2d}}.
$$
It follows that $\gamma^{\prime}(|\zeta|)\geq 0$.
\smallskip

\noindent\underline{We now prove that $\gamma$ is convex.} It is well-known (and easy to see by means of a convolution argument for regularising $\gamma$) that this is equivalent to proving that $\gamma^{\prime\prime}(t)\geq 0$ for all $t\in (0,+\infty)\setminus\{t_{0}\}$. Fix $t\notin\{t_{0},0\}$ and $\zeta\in\R^{2}$ such that $|\zeta|=t$. We rewrite \eqref{eq: App1} as
$$
\aligned
(D^{2}\Gamma)(\zeta)&=\gamma^{\prime\prime}(t)I_{\R^{2}}+\left(\gamma^{\prime}(t)-t\gamma^{\prime\prime}(t)\right)|\zeta|^{-1}\left[I_{\R^{2}}-\frac{\zeta}{|\zeta|}\otimes\frac{\zeta}{|\zeta|}\right],\\
&=\gamma^{\prime\prime}(t)I_{\R^{2}}+\left(\gamma^{\prime}(t)-t\gamma^{\prime\prime}(t)\right)(D^{2}F_{1})(\zeta),
\endaligned
$$
where $F_{1}(\zeta)=|\zeta|$. 

Therefore, for all $X\in\R^{2d}$ we have 
\begin{equation}\label{eq: sz1}
0\leq H^{A}_{\Gamma}[\zeta,X]=\gamma^{\prime\prime}(t)\sk{X}{\cM(A)X}_{\R^{2d}}+\left(\gamma^{\prime}(t)-t\gamma^{\prime\prime}(t)\right)H^{A}_{F_{1}}[\zeta;X].
\end{equation}
It follows from \eqref{eq: N Deltap} that $\Delta_{1}(A)\leq 0$. Hence, by Lemma~\ref{l: N prop of Deltap}~(\ref{L4 (v)}) we have
$$
\min_{|X|=1}H^{A}_{F_{1}}[\zeta;X]\leq 0.
$$
From this, \eqref{eq: sz1}, the fact that $\gamma^{\prime}\geq 0$, and the inequality $\sk{X}{\cM(A)X}_{\R^{2d}}\geq \lambda|X|^{2}$ we deduce that $\gamma^{\prime\prime}(t)\geq 0$. 
\smallskip

\noindent\underline{Convexity of $\Gamma$} is now clear and easily follows from the already proved properties of $\gamma$; for the reader's convenience, we give a complete proof. By using a standard convolution argument, it suffices to prove that $(D^{2}\Gamma)(\zeta)\geq 0$, for all $\zeta\notin \{|\zeta|=t_{0}\}\cup\{0\}$. For $\zeta\neq 0$ the two matrices $\tfrac{\zeta}{|\zeta|}\otimes\tfrac{\zeta}{|\zeta|}$ and $I_{\R^{2}}-\tfrac{\zeta}{|\zeta|}\otimes\tfrac{\zeta}{|\zeta|}$ (which represent the orthogonal projections on ${\rm{span}}\{\zeta\}$ and $\zeta^{\bot}$, respectively)  are positive semi-definite. Hence $D^{2}\Gamma(\zeta)\geq 0$, by \eqref{eq: App1} and the property of $\gamma$ that we have proved above.
\medskip

\noindent\underline{We now prove (ii).} Suppose that there exists $s>0$ such that $\gamma^{\prime}(s)=0$, then by item~(i) $\gamma^{\prime}(t)=0$ for all $t\in [0,s]$. Assume that $\gamma$ is nonconstant. Then,
$$
t_{1}:=\sup\{t>0: \gamma^{\prime}(t)=0\}\in (0,+\infty),
$$
$\gamma^{\prime}(t)>0$, for all $t>t_{1}$ and $\gamma^{\prime}(t_{1})=0$. For simplifying the proof, we assume that $t_{1}=t_{0}$. For $t>t_{0}$, define the function $v(t)=\log(\gamma^{\prime}(t))$. Then $\lim_{t\rightarrow t^{+}_{0} }v(t)=-\infty$ and  $v^{\prime}(t)=\frac{\gamma^{\prime\prime}(t)}{\gamma^{\prime}(t)}\geq 0$, for all $t>t_{0}$. It follows that
\begin{equation}\label{eq: APPB2}
\limsup_{t\rightarrow t^{+}_{0}}\frac{\gamma^{\prime\prime}(t)}{\gamma^{\prime}(t)}=+\infty.
\end{equation}

We now show that \eqref{eq: APPB2} implies $\Im  A=0$. Define the function
$$
p(t):=t\frac{\gamma^{\prime\prime}(t)}{\gamma^{\prime}(t)}+1\geq 1, \quad t>t_{0}.
$$
Then, by \eqref{eq: App1},
$$
(D^{2}\Gamma)(\zeta)=\frac{\gamma^{\prime}(|\zeta|)}{|\zeta|}\left[I_{\R^{2}}+(p(|\zeta|)-2)\,\frac{\zeta}{|\zeta|}\otimes\frac{\zeta}{|\zeta|}\right],\quad \forall |\zeta|>t_{0}.
$$
The formula above expresses the Hessian of the radial function $\Gamma$ in terms of Hessians of power functions. More specifically, by \eqref{eq: formula Hessian} we have
$$
(D^{2}\Gamma)(\zeta)=g(|\zeta|)\cdot(D^{2}F_{p(|\zeta|)})(\zeta),\quad \forall |\zeta|>t_{0},
$$
where
$$
g(t):=\frac{\gamma^{\prime}(t)}{p(t)t^{p(t)-1}}\geq 0, \quad t>t_{0}.
$$
Since $\Gamma$ is $A$-convex, it follows from the identity above that
$$
H^{A}_{F_{p(|\zeta|)}}[\zeta;X]\geq 0,\quad \forall X\in\R^{2d},\quad \forall |\zeta|>t_{0},
$$
or, equivalently (see Lemma~\ref{l: N prop of Deltap}~(\ref{L4 (v)})), $\Delta_{p(|\zeta|)}(A)\geq 0$, for all $|\zeta|>t_{0}$. Now by \eqref{eq: APPB2} we have $\sup_{|\zeta|>t_{0}}p(|\zeta|)=+\infty$, so by Lemma~\ref{l: N prop of Deltap}~(\ref{L4 (i)}) and (\ref{L4 (iii)}) we have $\Delta_{p}(A)\geq 0$, for all $p>1$. Hence, Lemma~\ref{l: N prop of Deltap}~(\ref{L4 (vi)}) implies that $\Im A=0$.
\end{proof}
\section{Flow regularity}
Let $(\cX,\mu)$ be a $\sigma$-finite measure space. Fix $p\geq 2$ and set $q=p/(p-1)$. Suppose that $(\exp(-t\cA))_{t>0}$ and $(\exp(-t\cB))_{t>0}$ are analytic and uniformly bounded both in $L^{p}(\cX,\mu)$ and $L^{q}(\cX,\mu)$. Let $\cQ=\cQ_{p,\delta}$ be the Nazarov-Treil Bellman function defined in \eqref{eq: N Bellman}. Fix $f,g\in (L^{p}\cap L^{q})(\cX,\mu)$. Consider the flow
$$
\cE(t)=\int_{\cX}\cQ(e^{-t\cA}f,e^{-t\cB}g),\quad t>0,
$$
where we omit the subscript $\cW$ (see Section~\ref{s: N heat-flow}). 

\begin{proposition}
Under the above assumptions, we have:
\begin{itemize}
\item[{\rm (a)}] $\cE\in C[0,+\infty)$;
\item[{\rm (b)}] $\cE\in C^{1}(0,+\infty)$ and
$$
-\cE^{\prime}(t)=2\Re\int_{\cX}\left((\partial_{\zeta}\cQ)(e^{-t\cA}f,e^{-t\cB}g)\cA e^{-t\cA}+(\partial_{\eta}\cQ)(e^{-t\cA}f,e^{-t\cB}g)\cB e^{-t\cB}g\right).
$$
\end{itemize}
\end{proposition}
\begin{proof}
\underline{We start with (a)}. We prove only the continuity at $0$ since the continuity at other points can be proved exactly in the same way, or 
it follows from item (b). Set
$$
F(t,x)=\cQ(e^{-t\cA}f(x),e^{-t\cB}g(x)).
$$
By the mean value theorem applied to $\cQ$,
$$
\aligned
\mod{F(t,x)-F(0,x)}\leq &\max\left\{\mod{(D\cQ)(\zeta,\eta)}: |\zeta|\leq \mod{e^{-t\cA}f(x)}+|f(x)|,\ |\eta|\leq \mod{e^{-t\cB}g(x)}+|g(x)|\right\}\\
&\times \sqrt{\mod{e^{-t\cA}f(x)-f(x)}^{2}+\mod{e^{-t\cB}g(x)-g(x)}^{2}}.
\endaligned
$$
Estimates \eqref{eq: N 5} immediately give
$$
\aligned
&\max\left\{\mod{(D\cQ)(\zeta,\eta)}: |\zeta|\leq \mod{e^{-t\cA}f(x)}+|f(x)|,\ |\eta|\leq \mod{e^{-t\cB}g(x)}+|g(x)|\right\}\\
&\leq C\max\left\{\left(|e^{-t\cA}f(x)|+|f(x)|\right)^{p-1},\left(|e^{-t\cA}g(x)|+|g(x)|\right)^{q-1}, |e^{-t\cA}g(x)|+|g(x)|\right\},
\endaligned
$$
where $C$ does not depend on $x$ and $t$. Now item (a) follows from H\"older's inequality and the strong continuity of the two semigroups in $L^{p}(\cX,\mu)$ and $L^{q}(\cX,\mu)$.
\medskip

\underline{We now prove item (b)}. Analyticity implies that there exist $C>0$ such that, for $r\in\{p,q\}$,
$$
\norm{\left(\frac{d}{d t}\right)^{k}e^{-t\cA}}{r}=\norm{\cA^{k}e^{-t\cA}}{r}\leq t^{-k}C^{k}k!,\quad \forall t>0,
$$
see, for example, \cite[Chapter~II, p.104]{EN}.
Fix $t_{0}>0$. Then there exists $\delta>0$ such that  
$$
e^{-t\cA}f=\sum^{+\infty}_{k=0}\frac{(-1)^{k}}{k!}(t-t_{0})^{k}\cA^{k}e^{-t_{0}\cA}f,\quad  \forall |t-t_{0}|\leq \delta
$$
where the series converges in $(L^{p}\cap L^{q})(\cX,\mu)$. Moreover,
$$
\sup_{|t-t_{0}|\leq \delta}\mod{e^{-t\cA}f(x)}\leq \sum^{+\infty}_{k=0}\frac{\delta^{k}}{k!}\mod{\cA^{k}e^{-t_{0}\cA}f(x)}\in (L^{p}\cap L^{q})(\cX,\mu)
$$
and similarly,
$$
\sup_{|t-t_{0}|\leq \delta}\mod{\cA e^{-t\cA}f(x)}\in (L^{p}\cap L^{q})(\cX,\mu).
$$
Possibly taking a smaller $\delta$, we also get
$$
\sup_{|t-t_{0}|\leq \delta}\mod{e^{-t\cB}g(x)}+\sup_{|t-t_{0}|\leq \delta}\mod{\cB e^{-t\cB}g(x)}\in (L^{p}\cap L^{q})(\cX,\mu).
$$
By using the powers series expansion of $e^{-t\cA}f$ and $e^{-t\cB}g$ one can also prove that each $e^{-t\cA}f$ and each $e^{-t\cB}g$ can be redefined in a set of measure zero, in such a manner that for almost every $x\in\Omega$ the functions $t\mapsto e^{-t\cA}f(x)$ is real-analytic on $(0,\infty)$, for a.e. $x\in\Omega$; see \cite[p. 72]{stein}.
Now item (b) follows from estimates \eqref{eq: N 5} and standard theorems of derivation and passage of the limit under the integral sign.
\end{proof}

\subsection*{Acknowledgements}

The first author was partially supported by the ``National Group for Mathematical Analysis, Probability and their Applications'' (GNAMPA-INdAM).

The second author was partially supported by the Ministry of Higher Education, Science and Technology of Slovenia (research program Analysis and Geometry, contract no. P1-0291).

\bibliographystyle{amsxport}
\bibliography{biblio_mixed}

\def\cprime{$'$} \def\cprime{$'$}
  \def\polhk#1{\setbox0=\hbox{#1}{\ooalign{\hidewidth
  \lower1.5ex\hbox{`}\hidewidth\crcr\unhbox0}}}
\begin{bibdiv}
\begin{biblist}

\bib{AdFour}{book}{
      author={Adams, Robert~A.},
      author={Fournier, John J.~F.},
       title={Sobolev spaces},
     edition={Second},
      series={Pure and Applied Mathematics (Amsterdam)},
   publisher={Elsevier/Academic Press, Amsterdam},
        date={2003},
      volume={140},
}

\bib{AgDoNi}{article}{
      author={Agmon, S.},
      author={Douglis, A.},
      author={Nirenberg, L.},
       title={Estimates near the boundary for solutions of elliptic partial
  differential equations satisfying general boundary conditions. {I}},
        date={1959},
     journal={Comm. Pure Appl. Math.},
      volume={12},
       pages={623\ndash 727},
}

\bib{AmbMaschain}{article}{
      author={Ambrosio, L.},
      author={Dal~Maso, G.},
       title={A general chain rule for distributional derivatives},
        date={1990},
     journal={Proc. Amer. Math. Soc.},
      volume={108},
      number={3},
       pages={691\ndash 702},
}

\bib{AuscherNuc}{article}{
      author={Auscher, Pascal},
       title={Regularity theorems and heat kernel for elliptic operators},
        date={1996},
     journal={J. London Math. Soc. (2)},
      volume={54},
      number={2},
       pages={284\ndash 296},
}

\bib{Auscher1}{article}{
      author={Auscher, Pascal},
       title={On necessary and sufficient conditions for {$L^p$}-estimates of
  {R}iesz transforms associated to elliptic operators on {$\Bbb R^n$} and
  related estimates},
        date={2007},
     journal={Mem. Amer. Math. Soc.},
      volume={186},
      number={871},
}

\bib{AMT}{article}{
      author={Auscher, Pascal},
      author={McIntosh, Alan},
      author={Tchamitchian, Philippe},
       title={Heat kernels of second order complex elliptic operators and
  applications},
        date={1998},
     journal={J. Funct. Anal.},
      volume={152},
      number={1},
       pages={22\ndash 73},
}

\bib{AuTc}{article}{
      author={Auscher, Pascal},
      author={Tchamitchian, Philippe},
       title={Square root problem for divergence operators and related topics},
        date={1998},
     journal={Ast\'erisque},
      number={249},
}

\bib{Bakry3}{incollection}{
      author={Bakry, Dominique},
       title={Sur l'interpolation complexe des semigroupes de diffusion},
        date={1989},
   booktitle={S\'eminaire de {P}robabilit\'es, {XXIII}},
      series={Lecture Notes in Math.},
      volume={1372},
   publisher={Springer, Berlin},
       pages={1\ndash 20},
}

\bib{BDFS}{article}{
      author={{Betancor}, Jorge~J.},
      author={{Dalmasso}, Estefan{\'\i}a},
      author={{Fari{\~n}a}, Juan~C.},
      author={{Scotto}, Roberto},
       title={{Bellman Functions and Dimension Free $L^p$-estimates for the
  Riesz Transforms in Bessel settings}},
        date={2018},
     journal={arXiv e-prints},
       pages={arXiv:1803.00789},
      eprint={1803.00789},
}

\bib{BK}{article}{
      author={Blunck, S.},
      author={Kunstmann, P.~C.},
       title={{C}alder\'on-{Z}ygmund theory for non-integral operators and the
  ${H}^{\infty}$ functional calculus},
        date={2003},
     journal={Rev. Mat. Iberoam.},
      number={19},
       pages={919\ndash 942},
}

\bib{BuDa}{article}{
      author={Burenkov, V.~I.},
      author={Davies, E.~B.},
       title={Spectral stability of the {N}eumann {L}aplacian},
        date={2002},
        ISSN={0022-0396},
     journal={J. Differential Equations},
      volume={186},
      number={2},
       pages={485\ndash 508},
         url={https://doi.org/10.1016/S0022-0396(02)00033-5},
      review={\MR{1942219}},
}

\bib{Bu1}{article}{
      author={Burkholder, D.~L.},
       title={{B}oundary value problems and sharp inequalities for martingale
  transforms},
        date={1984},
     journal={Ann. Prob.},
      volume={12},
      number={3},
       pages={647\ndash 702},
}

\bib{Bu2}{article}{
      author={Burkholder, D.~L.},
       title={{A} proof of {P}e\l czy\'nski's conjecture for the {H}aar
  system},
        date={1988},
     journal={Studia Math.},
      volume={1},
       pages={79\ndash 83},
}

\bib{Bu4}{incollection}{
      author={Burkholder, Donald~L.},
       title={{E}xplorations in martingale theory and its applications},
        date={1991},
   booktitle={\'{E}cole d'\'{E}t\'{e} de {P}robabilit\'es de {S}aint-{F}lour
  {XIX} ---1989},
      series={Lecture Notes in Math.},
      volume={1464},
   publisher={Springer, Berlin},
       pages={1\ndash 66},
}

\bib{CD-Riesz}{article}{
      author={Carbonaro, A.},
      author={Dragi\v{c}evi\'c, O.},
       title={{B}ellman function and dimension-free estimates in a theorem of
  {B}akry},
        date={2013},
     journal={J. Funct. Anal.},
      volume={265},
       pages={1085\ndash 1104},
}

\bib{CD-OU}{article}{
      author={Carbonaro, Andrea},
      author={Dragi\v{c}evi\'c, Oliver},
       title={Bounded holomorphic functional calculus for nonsymmetric
  {O}rnstein-{U}hlenbeck operators},
     journal={to appear in Ann. Sc. Norm. Super. Pisa Cl Sci.},
}

\bib{CD-DivForm}{article}{
      author={Carbonaro, Andrea},
      author={Dragi\v{c}evi\'c, Oliver},
       title={Convexity of power functions and bilinear embedding for
  divergence-form operators with complex coefficients},
     journal={to appear in J. Eur. Math. Soc. (JEMS)},
}

\bib{CD-mult}{article}{
      author={Carbonaro, Andrea},
      author={Dragi\v{c}evi\'c, Oliver},
       title={Functional calculus for generators of symmetric contraction
  semigroups},
        date={2017},
     journal={Duke Math. J.},
      volume={166},
      number={5},
       pages={937\ndash 974},
}

\bib{CiaMaz}{article}{
      author={Cialdea, A.},
      author={Maz'ya, V.},
       title={Criterion for the {$L^p$}-dissipativity of second order
  differential operators with complex coefficients},
        date={2005},
     journal={J. Math. Pures Appl. (9)},
      volume={84},
      number={8},
       pages={1067\ndash 1100},
}

\bib{Costabel}{article}{
      author={{Costabel}, Martin},
       title={{On the limit Sobolev regularity for Dirichlet and Neumann
  problems on Lipschitz domains}},
        date={2017Nov},
     journal={arXiv e-prints},
       pages={arXiv:1711.07179},
      eprint={1711.07179},
}

\bib{CourantHilbert}{book}{
      author={Courant, R.},
      author={Hilbert, D.},
       title={Methoden der {M}athematischen {P}hysik},
   publisher={Springer-Verlag, Berlin Heidelberg},
        date={1937},
      number={v. 2},
}

\bib{cowling}{article}{
      author={Cowling, M.},
       title={{H}armonic analysis on semigroups},
        date={1983},
     journal={Ann. Math. (2)},
      number={117},
       pages={267\ndash 283},
}

\bib{CDMY}{article}{
      author={Cowling, M.},
      author={Doust, I.},
      author={McIntosh, A.},
      author={Yagi, A.},
       title={{B}anach space operators with a bounded ${H}^\infty$ functional
  calculus},
        date={1996},
     journal={Austral. Math. Soc.},
      number={60},
       pages={51\ndash 89},
}

\bib{CrDe}{article}{
      author={Crouzeix, M.},
      author={Delyon, B.},
       title={{S}ome estimates for analytic functions of strip or sectorial
  operators},
        date={2003},
     journal={Arch. Math. (Basel)},
      volume={81},
      number={5},
       pages={559\ndash 566},
}

\bib{Dah}{article}{
      author={{Dahmani}, Kamilia},
       title={{Sharp dimension free bound for the Bakry-Riesz vector}},
        date={2016Nov},
     journal={arXiv e-prints},
       pages={arXiv:1611.07696},
      eprint={1611.07696},
}

\bib{DaSi1}{article}{
      author={Davies, E.~B.},
      author={Simon, B.},
       title={Spectral properties of {N}eumann {L}aplacian of horns},
        date={1992},
     journal={Geom. Funct. Anal.},
      volume={2},
      number={1},
       pages={105\ndash 117},
}

\bib{DHMP}{article}{
      author={Denk, Robert},
      author={Hieber, Matthias},
      author={Pr\"uss, Jan},
       title={{$\scr R$}-boundedness, {F}ourier multipliers and problems of
  elliptic and parabolic type},
        date={2003},
        ISSN={0065-9266},
     journal={Mem. Amer. Math. Soc.},
      volume={166},
      number={788},
       pages={viii+114},
         url={https://doi.org/10.1090/memo/0788},
      review={\MR{2006641}},
}

\bib{Dindos-Pipher2}{article}{
      author={{Dindo{\v{s}}}, Martin},
      author={{Pipher}, Jill},
       title={Boundary value problems for second order elliptic operators with
  complex coefficients},
        date={2018},
     journal={arXiv e-prints},
       pages={arXiv:1810.10366},
      eprint={1810.10366},
}

\bib{Dindos-Pipher3}{article}{
      author={{Dindo{\v{s}}}, Martin},
      author={{Pipher}, Jill},
       title={{Perturbation theory for solutions to second order elliptic
  operators with complex coefficients and the $L^p$ Dirichlet problem}},
        date={2018},
     journal={arXiv e-prints},
       pages={arXiv:1805.08614},
      eprint={1805.08614},
}

\bib{Dindos-Pipher}{article}{
      author={Dindo\v{s}, Martin},
      author={Pipher, Jill},
       title={Regularity theory for solutions to second order elliptic
  operators with complex coefficients and the {$L^p$} {D}irichlet problem},
        date={2019},
     journal={Adv. Math.},
      volume={341},
       pages={255\ndash 298},
         url={https://doi.org/10.1016/j.aim.2018.07.035},
}

\bib{DomPe}{article}{
      author={Domelevo, Komla},
      author={Petermichl, Stefanie},
       title={Sharp {$L^p$} estimates for discrete second order {R}iesz
  transforms},
        date={2014},
     journal={Adv. Math.},
      volume={262},
       pages={932\ndash 952},
}

\bib{Dore}{incollection}{
      author={Dore, Giovanni},
       title={{$L^p$} regularity for abstract differential equations},
        date={1993},
   booktitle={Functional analysis and related topics, 1991 ({K}yoto)},
      series={Lecture Notes in Math.},
      volume={1540},
   publisher={Springer, Berlin},
       pages={25\ndash 38},
}

\bib{DoreVenni}{article}{
      author={Dore, Giovanni},
      author={Venni, Alberto},
       title={On the closedness of the sum of two closed operators},
        date={1987},
     journal={Math. Z.},
      volume={196},
      number={2},
       pages={189\ndash 201},
}

\bib{DraVol}{article}{
      author={Dragi{\v c}evi\'c, Oliver},
      author={Volberg, Alexander},
       title={Bellman function, {L}ittlewood-{P}aley estimates and asymptotics
  for the {A}hlfors-{B}eurling operator in {$L^p(\Bbb C)$}},
        date={2005},
     journal={Indiana Univ. Math. J.},
      volume={54},
      number={4},
       pages={971\ndash 995},
}

\bib{DV}{article}{
      author={Dragi{\v{c}}evi{\'c}, Oliver},
      author={Volberg, Alexander},
       title={Bellman functions and dimensionless estimates of
  {L}ittlewood-{P}aley type},
        date={2006},
     journal={J. Operator Theory},
      volume={56},
      number={1},
       pages={167\ndash 198},
}

\bib{Dv-kato}{article}{
      author={Dragi{\v{c}}evi{\'c}, Oliver},
      author={Volberg, Alexander},
       title={Bilinear embedding for real elliptic differential operators in
  divergence form with potentials},
        date={2011},
     journal={J. Funct. Anal.},
      volume={261},
      number={10},
       pages={2816\ndash 2828},
}

\bib{DV-Sch}{article}{
      author={Dragi{\v{c}}evi{\'c}, Oliver},
      author={Volberg, Alexander},
       title={Linear dimension-free estimates in the embedding theorem for
  {S}chr\"odinger operators},
        date={2012},
     journal={J. Lond. Math. Soc. (2)},
      volume={85},
      number={1},
       pages={191\ndash 222},
}

\bib{Egert}{article}{
      author={Egert, Moritz},
       title={{$L^p$}-estimates for the square root of elliptic systems with
  mixed boundary conditions},
        date={2018},
     journal={J. Differential Equations},
      volume={265},
      number={4},
       pages={1279\ndash 1323},
}

\bib{Egert2018}{article}{
      author={Egert, Moritz},
       title={On $p$-elliptic divergence form operators and holomorphic
  semigroups},
        date={2018},
     journal={arXiv e-prints},
       pages={arXiv:1812.09154},
      eprint={1812.09154},
}

\bib{EN}{book}{
      author={Engel, Klaus-Jochen},
      author={Nagel, Rainer},
       title={One-parameter semigroups for linear evolution equations},
      series={Graduate Texts in Mathematics},
   publisher={Springer-Verlag, New York},
        date={2000},
      volume={194},
}

\bib{EvHa}{article}{
      author={Evans, W.~D.},
      author={Harris, D.~J.},
       title={On the approximation numbers of {S}obolev embeddings for
  irregular domains},
        date={1989},
     journal={Quart. J. Math. Oxford Ser. (2)},
      volume={40},
      number={157},
       pages={13\ndash 42},
}

\bib{GigaSohr}{article}{
      author={Giga, Yoshikazu},
      author={Sohr, Hermann},
       title={Abstract {$L^p$} estimates for the {C}auchy problem with
  applications to the {N}avier-{S}tokes equations in exterior domains},
        date={1991},
        ISSN={0022-1236},
     journal={J. Funct. Anal.},
      volume={102},
      number={1},
       pages={72\ndash 94},
         url={https://doi.org/10.1016/0022-1236(91)90136-S},
      review={\MR{1138838}},
}

\bib{Gris}{book}{
      author={Grisvard, P.},
       title={Elliptic problems in nonsmooth domains},
      series={Monographs and Studies in Mathematics},
   publisher={Pitman (Advanced Publishing Program), Boston, MA},
        date={1985},
      volume={24},
}

\bib{Haase}{book}{
      author={Haase, Markus},
       title={The functional calculus for sectorial operators},
      series={Operator Theory: Advances and Applications},
   publisher={Birkh\"auser Verlag, Basel},
        date={2006},
      volume={169},
}

\bib{PioPekka}{article}{
      author={Haj{\l a}sz, Piotr},
      author={Koskela, Pekka},
       title={Isoperimetric inequalities and imbedding theorems in irregular
  domains},
        date={1998},
     journal={J. London Math. Soc. (2)},
      volume={58},
      number={2},
       pages={425\ndash 450},
}

\bib{H-DJKR}{article}{
      author={Haller-Dintelmann, Robert},
      author={Jonsson, Alf},
      author={Knees, Dorothee},
      author={Rehberg, Joachim},
       title={Elliptic and parabolic regularity for second-order divergence
  operators with mixed boundary conditions},
        date={2016},
     journal={Math. Methods Appl. Sci.},
      volume={39},
      number={17},
       pages={5007\ndash 5026},
}

\bib{HeSeSi}{article}{
      author={Hempel, Rainer},
      author={Seco, Luis~A.},
      author={Simon, Barry},
       title={The essential spectrum of {N}eumann {L}aplacians on some bounded
  singular domains},
        date={1991},
     journal={J. Funct. Anal.},
      volume={102},
      number={2},
       pages={448\ndash 483},
}

\bib{HMMcl}{article}{
      author={Hofmann, Steve},
      author={Mayboroda, Svitlana},
      author={McIntosh, Alan},
       title={Second order elliptic operators with complex bounded measurable
  coefficients in {$L^p$}, {S}obolev and {H}ardy spaces},
        date={2011},
     journal={Ann. Sci. \'Ec. Norm. Sup\'er. (4)},
      volume={44},
      number={5},
}

\bib{JeKe}{incollection}{
      author={Jerison, David},
      author={Kenig, Carlos~E.},
       title={The functional calculus for the {L}aplacian on {L}ipschitz
  domains},
        date={1989},
   booktitle={Journ\'ees ``\'equations aux {D}\'eriv\'ees {P}artielles''
  ({S}aint {J}ean de {M}onts, 1989)},
   publisher={\'Ecole Polytech., Palaiseau},
       pages={Exp.\ No.\ IV, 10},
}

\bib{Kalton}{incollection}{
      author={Kalton, N.~J.},
       title={A remark on sectorial operators with an {$H^\infty$}-calculus},
        date={2003},
   booktitle={Trends in {B}anach spaces and operator theory ({M}emphis, {TN},
  2001)},
      series={Contemp. Math.},
      volume={321},
   publisher={Amer. Math. Soc., Providence, RI},
       pages={91\ndash 99},
}

\bib{KW}{article}{
      author={Kalton, N.~J.},
      author={Weis, L.},
       title={The {$H^\infty$}-calculus and sums of closed operators},
        date={2001},
     journal={Math. Ann.},
      volume={321},
      number={2},
       pages={319\ndash 345},
}

\bib{Kat}{book}{
      author={Kato, Tosio},
       title={Perturbation theory for linear operators},
     edition={Second},
   publisher={Springer-Verlag, Berlin-New York},
        date={1976},
        note={Grundlehren der Mathematischen Wissenschaften, Band 132},
}

\bib{KR}{article}{
      author={Kriegler, Ch.},
       title={Analyticity angle for non-commutative diffusion semigroups},
        date={2011},
     journal={J. Lond. Math. Soc. (2)},
      volume={83},
      number={1},
       pages={168\ndash 186},
}

\bib{KuW}{incollection}{
      author={Kunstmann, Peer~C.},
      author={Weis, Lutz},
       title={Maximal {$L_p$}-regularity for parabolic equations, {F}ourier
  multiplier theorems and {$H^\infty$}-functional calculus},
        date={2004},
   booktitle={Functional analytic methods for evolution equations},
      series={Lecture Notes in Math.},
      volume={1855},
   publisher={Springer, Berlin},
       pages={65\ndash 311},
}

\bib{Kunstmann3}{article}{
      author={Kunstmann, Peer~Christian},
       title={Uniformly elliptic operators with maximal {$L_p$}-spectrum in
  planar domains},
        date={2001},
     journal={Arch. Math. (Basel)},
      volume={76},
      number={5},
       pages={377\ndash 384},
}

\bib{kunstmann2}{article}{
      author={Kunstmann, Peer~Christian},
       title={{$L_p$}-spectral properties of the {N}eumann {L}aplacian on
  horns, comets and stars},
        date={2002},
     journal={Math. Z.},
      volume={242},
      number={1},
       pages={183\ndash 201},
}

\bib{KS}{article}{
      author={Kunstmann, Peer~Christian},
      author={{\v{S}}trkalj, {\v{Z}}eljko},
       title={{$H^\infty$}-calculus for submarkovian generators},
        date={2003},
     journal={Proc. Amer. Math. Soc.},
      volume={131},
      number={7},
       pages={2081\ndash 2088 (electronic)},
}

\bib{Leonichain}{article}{
      author={Leoni, Giovanni},
      author={Morini, Massimiliano},
       title={Necessary and sufficient conditions for the chain rule in
  {$W^{1,1}_{\rm loc}(\Bbb R^N;\Bbb R^d)$} and {${\rm BV}_{\rm loc}(\Bbb
  R^N;\Bbb R^d)$}},
        date={2007},
     journal={J. Eur. Math. Soc. (JEMS)},
      volume={9},
      number={2},
       pages={219\ndash 252},
}

\bib{lp}{article}{
      author={Liskevich, V.~A.},
      author={Perel{\cprime}muter, M.~A.},
       title={Analyticity of sub-{M}arkovian semigroups},
        date={1995},
     journal={Proc. Amer. Math. Soc.},
      volume={123},
      number={4},
       pages={1097\ndash 1104},
}

\bib{Lunardib1}{book}{
      author={Lunardi, Alessandra},
       title={Analytic semigroups and optimal regularity in parabolic
  problems},
      series={Progress in Nonlinear Differential Equations and their
  Applications},
   publisher={Birkh\"auser Verlag, Basel},
        date={1995},
      volume={16},
}

\bib{MaSp}{article}{
      author={Mauceri, Giancarlo},
      author={Spinelli, Micol},
       title={Riesz transforms and spectral multipliers of the
  {H}odge-{L}aguerre operator},
        date={2015},
     journal={J. Funct. Anal.},
      volume={269},
      number={11},
       pages={3402\ndash 3457},
}

\bib{Mc}{incollection}{
      author={McIntosh, Alan},
       title={Operators which have an {$H_\infty$} functional calculus},
        date={1986},
   booktitle={Miniconference on operator theory and partial differential
  equations ({N}orth {R}yde, 1986)},
      series={Proc. Centre Math. Anal. Austral. Nat. Univ.},
      volume={14},
   publisher={Austral. Nat. Univ., Canberra},
       pages={210\ndash 231},
}

\bib{McY}{incollection}{
      author={McIntosh, Alan},
      author={Yagi, Atsushi},
       title={Operators of type {$\omega$} without a bounded {$H_\infty$}
  functional calculus},
        date={1990},
   booktitle={Miniconference on {O}perators in {A}nalysis ({S}ydney, 1989)},
      series={Proc. Centre Math. Anal. Austral. Nat. Univ.},
      volume={24},
   publisher={Austral. Nat. Univ., Canberra},
       pages={159\ndash 172},
}

\bib{NTV}{article}{
      author={Nazarov, F.},
      author={Treil, S.},
      author={Volberg, A.},
       title={The {B}ellman functions and two-weight inequalities for {H}aar
  multipliers},
        date={1999},
     journal={J. Amer. Math. Soc.},
      volume={12},
      number={4},
       pages={909\ndash 928},
}

\bib{NTV1}{incollection}{
      author={Nazarov, F.},
      author={Treil, S.},
      author={Volberg, A.},
       title={Bellman function in stochastic control and harmonic analysis},
        date={2001},
   booktitle={Systems, approximation, singular integral operators, and related
  topics ({B}ordeaux, 2000)},
      series={Oper. Theory Adv. Appl.},
      volume={129},
   publisher={Birkh\"auser, Basel},
       pages={393\ndash 423},
}

\bib{NT}{article}{
      author={Nazarov, F.~L.},
      author={Tre{\u\i}l{\cprime}, S.~R.},
       title={The hunt for a {B}ellman function: applications to estimates for
  singular integral operators and to other classical problems of harmonic
  analysis},
        date={1996},
     journal={Algebra i Analiz},
      volume={8},
      number={5},
       pages={32\ndash 162},
}

\bib{Nittka}{article}{
      author={Nittka, Robin},
       title={Projections onto convex sets and {$L^p$}-quasi-contractivity of
  semigroups},
        date={2012},
     journal={Arch. Math. (Basel)},
      volume={98},
      number={4},
       pages={341\ndash 353},
}

\bib{Os}{book}{
      author={Os{\polhk{e}}kowski, Adam},
       title={Sharp martingale and semimartingale inequalities},
      series={Mathematics Institute of the Polish Academy of Sciences.
  Mathematical Monographs (New Series)},
   publisher={Birkh\"auser/Springer Basel AG, Basel},
        date={2012},
      volume={72},
}

\bib{Ouh1992}{article}{
      author={Ouhabaz, El-Maati},
       title={{$L^\infty$}-contractivity of semigroups generated by sectorial
  forms},
        date={1992},
     journal={J. London Math. Soc. (2)},
      volume={46},
      number={3},
       pages={529\ndash 542},
}

\bib{Ouh1996}{article}{
      author={Ouhabaz, El-Maati},
       title={Invariance of closed convex sets and domination criteria for
  semigroups},
        date={1996},
     journal={Potential Anal.},
      volume={5},
      number={6},
       pages={611\ndash 625},
}

\bib{O}{book}{
      author={Ouhabaz, El~Maati},
       title={Analysis of heat equations on domains},
      series={London Mathematical Society Monographs Series},
   publisher={Princeton University Press, Princeton, NJ},
        date={2005},
      volume={31},
}

\bib{pazy}{book}{
      author={Pazy, A.},
       title={Semigroups of linear operators and applications to partial
  differential equations},
      series={Applied Mathematical Sciences},
   publisher={Springer-Verlag, New York},
        date={1983},
      volume={44},
}

\bib{PSW}{article}{
      author={Petermichl, Stefanie},
      author={Slavin, Leonid},
      author={Wick, Brett~D.},
       title={New estimates for the {B}eurling-{A}hlfors operator on
  differential forms},
        date={2011},
     journal={J. Operator Theory},
      volume={65},
      number={2},
       pages={307\ndash 324},
}

\bib{PV}{article}{
      author={Petermichl, Stefanie},
      author={Volberg, Alexander},
       title={Heating of the {A}hlfors-{B}eurling operator: weakly quasiregular
  maps on the plane are quasiregular},
        date={2002},
     journal={Duke Math. J.},
      volume={112},
      number={2},
       pages={281\ndash 305},
}

\bib{PrussSohr}{article}{
      author={Pr\"uss, Jan},
      author={Sohr, Hermann},
       title={On operators with bounded imaginary powers in {B}anach spaces},
        date={1990},
        ISSN={0025-5874},
     journal={Math. Z.},
      volume={203},
      number={3},
       pages={429\ndash 452},
         url={https://doi.org/10.1007/BF02570748},
      review={\MR{1038710}},
}

\bib{stein}{book}{
      author={Stein, Elias~M.},
       title={Topics in harmonic analysis related to the {L}ittlewood-{P}aley
  theory.},
      series={Annals of Mathematics Studies, No. 63},
   publisher={Princeton University Press, Princeton, N.J.; University of Tokyo
  Press, Tokyo},
        date={1970},
}

\bib{Elst2019}{article}{
      author={{ter Elst}, A.~F.~M.},
      author={{Haller-Dintelmann}, R.},
      author={{Rehberg}, J.},
      author={{Tolksdorf}, P.},
       title={On the ${L}^{p}$-theory for second-order elliptic operators in
  divergence form with complex coefficients},
        date={2019},
     journal={arXiv e-prints},
       pages={arXiv:1903.06692},
      eprint={1903.06692},
}

\bib{Tolks}{article}{
      author={Tolksdorf, Patrick},
       title={{$\R$}-sectoriality of higher-order elliptic systems on general
  bounded domains},
        date={2018},
     journal={J. Evol. Equ.},
      volume={18},
      number={2},
       pages={323\ndash 349},
}

\bib{NV}{article}{
      author={Volberg, A.},
      author={Nazarov, F.},
       title={Heat extension of the {B}eurling operator and estimates for its
  norm},
        date={2003},
     journal={Algebra i Analiz},
      volume={15},
      number={4},
       pages={142\ndash 158},
}

\bib{V}{article}{
      author={Volberg, Alexander},
       title={Bellman approach to some problems in harmonic analysis,},
        date={2002},
     journal={Expos{\'e} n$^\circ$ XIX, {E}cole {P}olyt{\'e}chnique},
       pages={14},
}

\bib{Weis1}{article}{
      author={Weis, Lutz},
       title={Operator-valued {F}ourier multiplier theorems and maximal
  {$L_p$}-regularity},
        date={2001},
     journal={Math. Ann.},
      volume={319},
      number={4},
       pages={735\ndash 758},
}

\bib{W}{article}{
      author={Wittwer, Janine},
       title={Survey article: a user's guide to {B}ellman functions},
        date={2011},
     journal={Rocky Mountain J. Math.},
      volume={41},
      number={3},
       pages={631\ndash 661},
}

\bib{Wrobel1}{article}{
      author={Wr\'{o}bel, Bazej},
       title={Dimension-free {$L^p$} estimates for vectors of {R}iesz
  transforms associated with orthogonal expansions},
        date={2018},
     journal={Anal. PDE},
      volume={11},
      number={3},
       pages={745\ndash 773},
}

\bib{Yan}{article}{
      author={Yan, Min},
       title={Extension of convex function},
        date={2014},
     journal={J. Convex Anal.},
      volume={21},
      number={4},
       pages={965\ndash 987},
}

\end{biblist}
\end{bibdiv}

\end{document}